\documentclass[11pt,a4paper]{article}
\usepackage{amssymb,amsmath,graphicx}
\newtheorem{definition}{Definition}[section]
\newtheorem{proposition}[definition]{Proposition}
\newtheorem{lemma}[definition]{Lemma}
\newtheorem{theorem}[definition]{Theorem}
\newtheorem{remark}[definition]{Remark}
\newtheorem{corollary}[definition]{Corollary}
\newenvironment{proof}{\begin{trivlist}\item[] \textit{Proof.}}
                     {\hspace*{\fill} $\square$\end{trivlist}}
\newenvironment{pfofthm42}{\begin{trivlist}\item[] \textit{Proof of Theorem \ref{uni and attr}.}}
                     {\hspace*{\fill} $\square$\end{trivlist}}

\usepackage{comment}
\usepackage{stmaryrd}

\newcommand{\bE}{\mathbb{E}}

\newcommand{\bN}{\mathbb{N}}
\newcommand{\bP}{\mathbb{P}}
\newcommand{\bR}{\mathbb{R}}
\newcommand{\bT}{\mathbb{T}}
\newcommand{\bU}{\mathbb{U}}

\newcommand{\cT}{\mathcal{T}}

\begin{document}

\title{A recursive distribution equation for the stable tree\thanks{University of Oxford, niccheekn@gmail.com, franz.rembart@gmail.com, winkel@stats.ox.ac.uk}}

\author{%
  Nicholas~Chee%\thanks{niccheekn@gmail.com} 
  \and %% remove this line and below if single author
  Franz~Rembart%\thanks{franz.rembart@gmail.com}
  \and Matthias~Winkel%\thanks{winkel@stats.ox.ac.uk}
}

\maketitle

\begin{abstract}
  \noindent We provide a new characterisation of Duquesne and Le Gall's $\alpha$-stable tree, \vspace{-0.1cm} $\alpha\!\in\!(1,2]$,\linebreak as the solution of a recursive 
  distribution equation (RDE) of the form $\mathcal{T}\! \overset{d}{=}\! g\left(\xi,\!\mathcal{T}_i, i\!\geq\!0\right)$, where $g$ is a concatenation 
  operator, $\xi = \left(\xi_i, i\geq 0\right)$ a sequence of scaling factors, $\mathcal{T}_i$, $i \geq 0$, and $\mathcal{T}$ are i.i.d.\ 
  trees independent of $\xi$. This generalises a version of the well-known characterisation of the Brownian Continuum Random Tree due to Aldous, 
  Albenque and Goldschmidt. By relating to previous results on a rather different class of RDE, we explore the present RDE and obtain for a large 
  class of similar RDEs that the fixpoint is unique (up to multiplication by a constant) and attractive. 
  
  
\end{abstract}

\section{Introduction}

\label{Introduction}

$\mathbb{R}$-trees, constitute a class of loop-free length spaces which frequently arise as scaling limits of many discrete trees 
\cite{d06}. In their own right, $\mathbb{R}$-trees have diverse applications from rough path integration theory \cite{hambly2006uniqueness} to
phylogenetic models \cite{felsenstein2004inferring}. Following Aldous's introduction of the Brownian Continuum Random Tree (BCRT) \cite{a91a,a91b,a93}, significant attention turned to random $\mathbb{R}$-trees. Naturally, the BCRT manifests in the asymptotics of discrete tree-like structures, including uniform random labelled trees \cite{a91a,a93} and critical Galton--Watson trees with finite offspring variance \cite{a91a}. Bewilderingly, recent applications of the BCRT have surpassed objects not overtly tree-like, for example, random recursive triangulations \cite{curien2011random}, random planar quadrangulations \cite{miermont2013brownian}, and Liouville quantum gravity \cite{duplantier2014liouville}. 

The BCRT was generalised by Duquesne and Le Gall's $\alpha$-stable trees \cite{duquesne2002random,duquesne2005probabilistic}, parameterised by $\alpha \in (1,2]$. The $\alpha$-stable trees are themselves a special case of Le Gall and Le Jan's L\'{e}vy trees \cite{g98}, representing the genealogies of continuous-state branching processes with branching mechanism $\psi(\lambda)=\lambda^{\alpha}$. When $\alpha = 2$, we recover the BCRT. Akin to the BCRT, the family of $\alpha$-stable trees constitutes all possible scaling limits of Galton--Watson trees, conditioned on the total progeny, whose offspring distribution lies in the domain of attraction of an $\alpha$-stable law \cite{duquesne2003limit}. Likewise, $\alpha$-stable trees emerge in scaling limits of numerous discrete tree structures, e.g., vertex-cut Galton--Watson trees \cite{dieuleveut2015vertex} and conditioned stable L\'{e}vy forests \cite{chaumont2007genealogy}. Pursuing a dedicated approach with L\'{e}vy processes gives links to superprocesses \cite{duquesne2002random,g98}, and beta-coalescents in genetic models \cite{abraham2013beta,berestycki2007beta}. Particular aspects of $\alpha$-stable trees, such as, invariance under uniform re-rooting \cite{haas2009spinal}, Hausdorff and packing measures \cite{duquesne2012exact,duquesne2005hausdorff,
duquesne2005probabilistic}, spectral dimensions \cite{croydonhambly2010spectral}, heights and diameters \cite{duquesne2015decomposition}, and an embedding property of stable trees \cite{curien2013stable}, have also been closely studied. 

We wish to emphasise a crucial self-similarity property of $\alpha$-stable trees. This property plausibly explains the prevalence of $\alpha$-stable trees in such diverse contexts, especially in problems of a recursive nature. Decomposing an $\alpha$-stable tree above a certain height or at appropriate nodes results in the connected components after decomposition forming  rescaled independent copies of the original tree. This observation was first formalised by Miermont \cite{miermont2003self,miermont2005self}, building upon Bertoin's self-similar fragmentation theory \cite{bertoin2002self}. 

In this paper, we express the self-similarity of the $\alpha$-stable tree by a new \textit{recursive distribution equation} (RDE) in the setting of Aldous and Bandyopadhyay's survey paper \cite{ab05}. Given a random variable $\mathcal{T}$ valued in a Polish metric space $(\mathbb{T},d)$, an RDE is a stochastic equation of the form
\begin{equation*}
\mathcal{T} \overset{d}{=} g \left(\xi,\mathcal{T}_i, i\geq 0\right) \quad \textnormal{ on } \mathbb{T},
\end{equation*}
where $(\mathcal{T}_i, i \geq 0)$ are i.i.d.\ and distributed as $\mathcal T$, $g$ is a measurable mapping, and $\xi$ is independent of $\left(\mathcal{T}_i, i\geq 0\right)$. RDEs are pertinent in various contexts with recursive structures, including Galton--Watson branching processes \cite{ab05}, Poisson weighted infinite trees \cite{bandyopadhyay2003bivariate}, and Quicksort algorithms \cite{rr01}. 

RDEs have been employed in the recursive construction of the BCRT by Albenque and Goldschmidt \cite{ag15}, recursively concatenating three
rescaled trees at a single point. Broutin and Sulzbach \cite{broutin2016self} extended this to further recursive combinatorial structures and 
weighted $\mathbb{R}$-trees under a finite concatenation operation. Rembart and Winkel \cite{rw16} did similarly with $\mathbb{R}$-trees under a 
different operation that concatenates a countable (possibly infinite) number of rescaled trees to a branch/spine. See Figure 
\ref{fig:decs}.

\begin{figure}[t]
  $\;$\hfill\includegraphics[width=6.5cm]{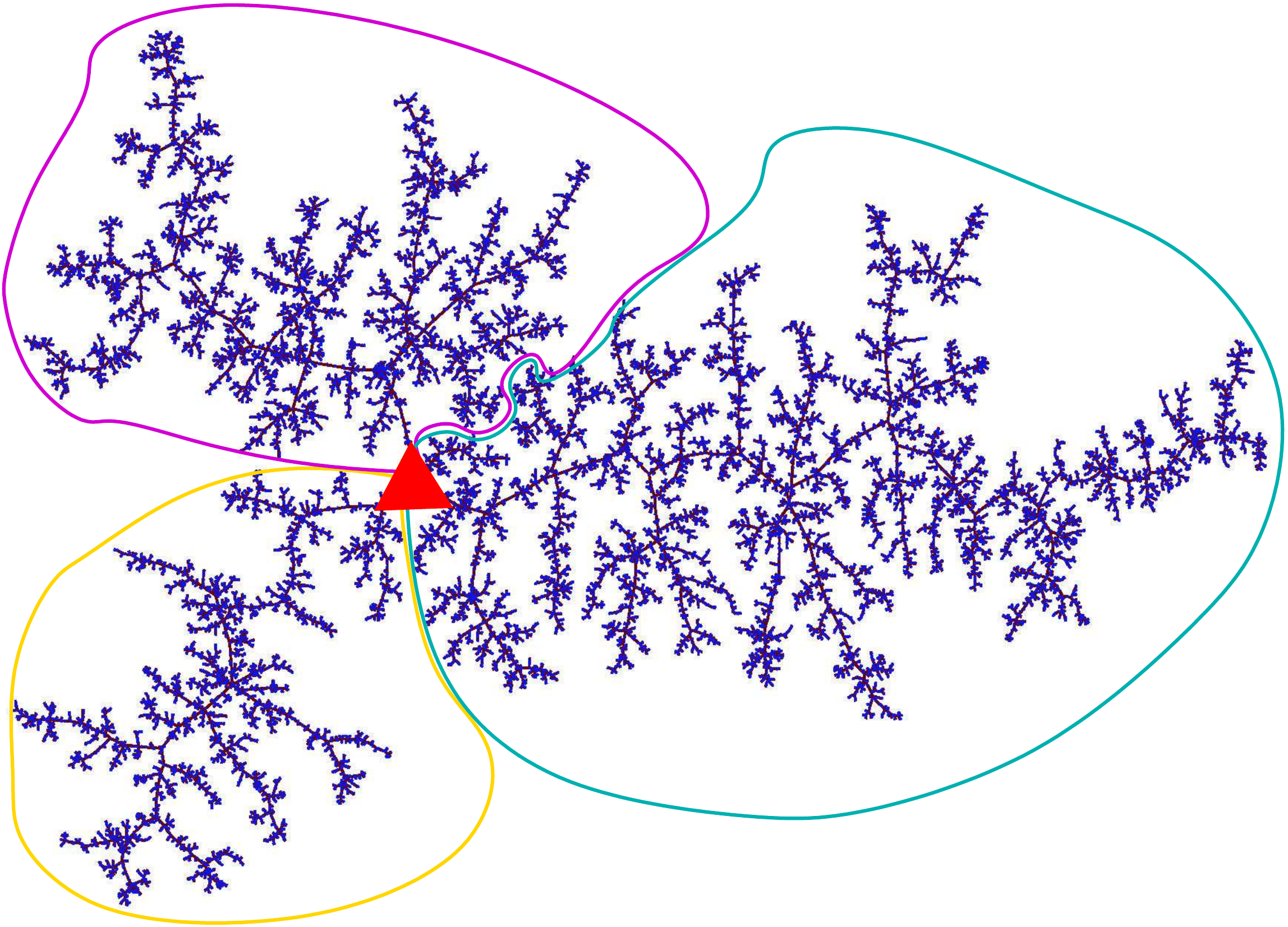}\hfill \includegraphics[width=6.5cm]{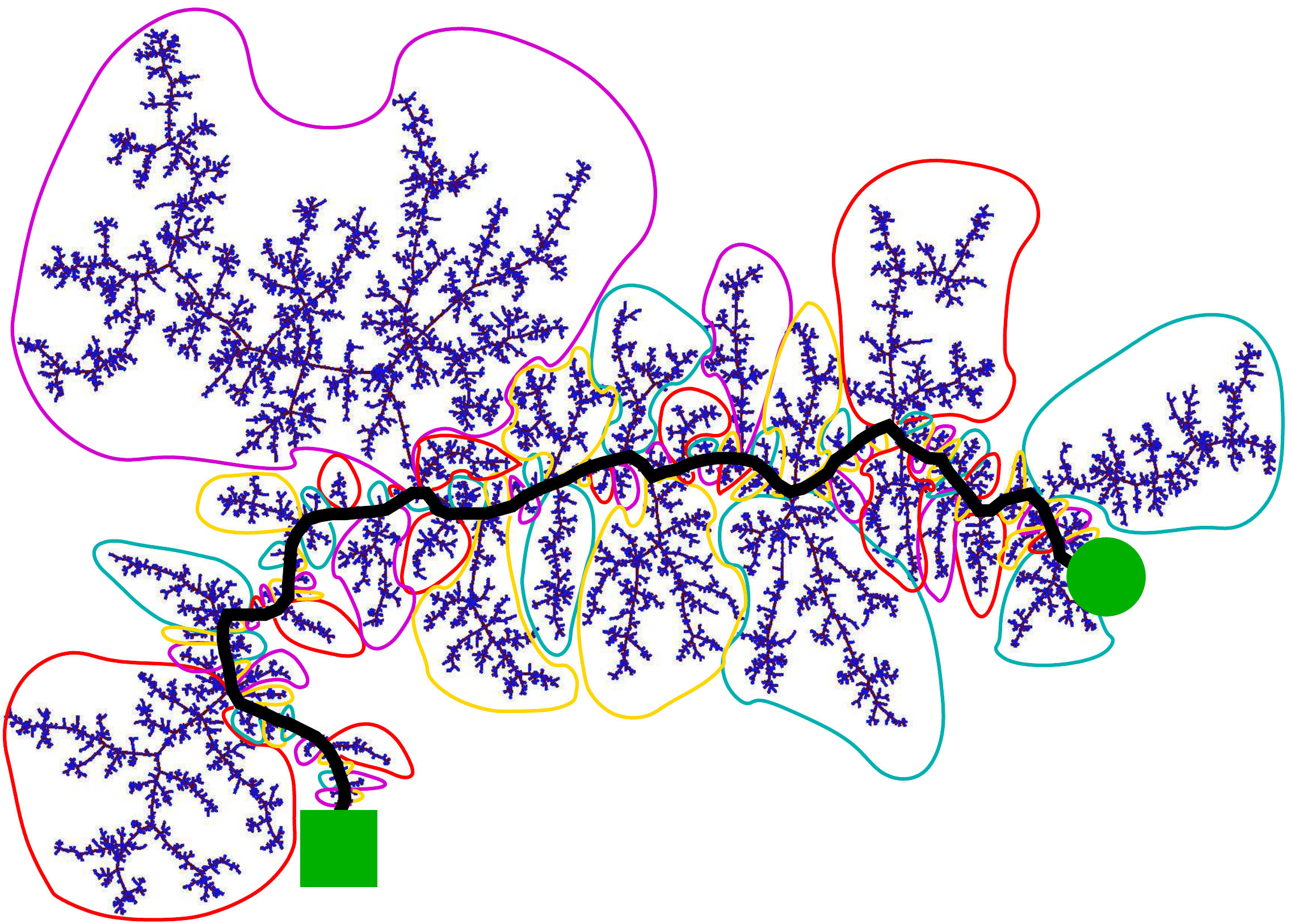}\hfill$\;$
  \caption{RDEs derived from the decomposition of the BCRT (simulation courtesy of Igor Kortchemski) around a branchpoint 
  (the red triangle) into three parts and along the spine from the root (green square) to a random leaf (green circle) into ``infinitely many'' parts.}
  \label{fig:decs}
\end{figure}

In this paper, we consider as $g$ the operation that concatenates at a single point a countable number of $\mathbb{R}$-trees 
$\mathcal{T}_i \overset{d}{=} \mathcal{T}$, rescaled by $\xi_i\geq 0$, $i\geq 0$, respectively, seeking to obtain a version of $\mathcal{T}$. 
Theorem \ref{RDE for alpha-stable tree THM} shows that the law of the $\alpha$-stable tree is a fixpoint solution of an RDE of this type. This
is illustrated in Figure \ref{fig:decs2}. 
Our primary argument appeals to Marchal's random growth algorithm \cite{marchal2008note}, which provides a recursive method of constructing 
$\alpha$-stable trees as a scaling limit. To explore the uniqueness of this solution (up to rescaling distances by a constant) we first observe
that we require 
certain finite height moments. In the absence of this condition, further solutions can be obtained, for example, by decorating the $\alpha$-stable tree with massless branches, see Remark \ref{RDE alpha counterexample Prop}. 

\begin{figure}[t]
  $\;$\hfill
  \includegraphics[width=13cm]{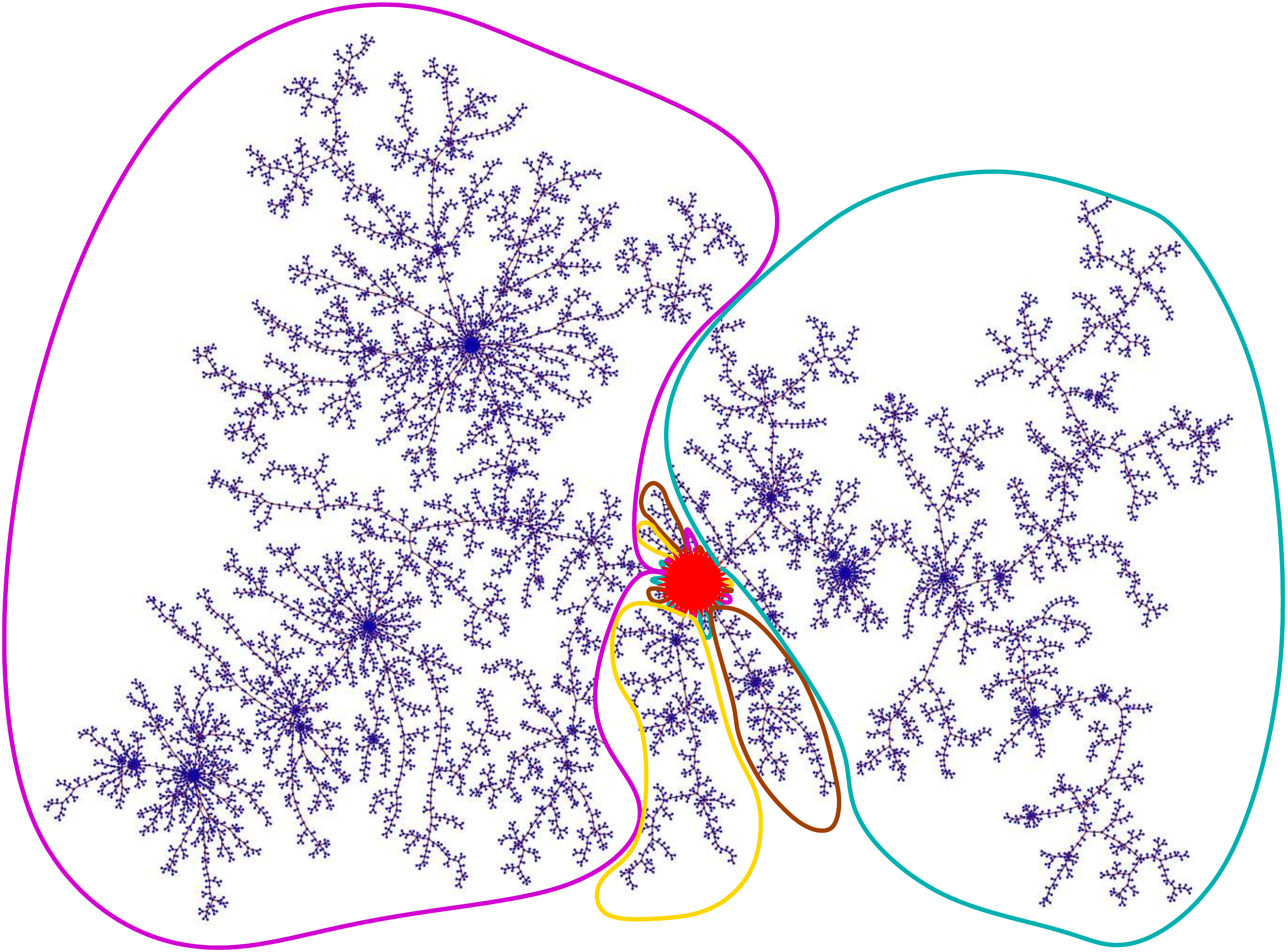}
  \hfill$\;$
  \caption{RDEs derived from the decomposition of a stable tree (simulation courtesy of Igor Kortchemski) around a branchpoint (the red star) into ``infinitely many'' parts.\vspace{-0.2cm}}
  \label{fig:decs2}
\end{figure}

Let us explore our approach to uniqueness and attraction in the context of the literature. While our results closely resemble  
\cite{ag15} and \cite{broutin2016self} for the (binary) BCRT and other finitely branching structures, our methods are rather different. 
Indeed, our results extend finite concatenation operations to handle trees such as the $\alpha$-stable trees, whose branch points are of countably 
infinite multiplicity. Extending their uniqueness and attraction results is not straightforward using the methods of \cite{ag15,broutin2016self}. 
On the other hand, \cite{rw16} 
presents an RDE for which the law of the $\alpha$-stable tree is a unique and attractive fixpoint, but the concatenation approach of employing 
strings of beads (weighted intervals) and bead-splitting processes of \cite{pitman2015regenerative} is different. Our RDEs only require countably 
infinite weight sequences such as Poisson--Dirichlet sequences and gives a less technical recursive construction of $\alpha$-stable trees that 
elucidates how mass partitions in $\alpha$-stable trees relate to urn models and partition-valued processes. 

Specifically, we prove the self-similarity property of $\alpha$-stable trees decomposed at a branch point solely via the recursive nature of 
Marchal's algorithm, without need for Miermont's fragmentation tree theory \cite{miermont2005self}. To prove our uniqueness and attraction result,
Theorem \ref{uni and attr}, we establish a connection between the two types of RDE, which effectively breaks down the proofs here into a 
one-dimensional martingale argument, the uniqueness and attraction of the RDE of \cite{rw16} and a tightness argument that again builds on 
\cite{rw16} by constructing an auxiliary dominating CRT.%\medskip

The structure of this paper is as follows. In Section \ref{Preliminaries}, we state background results on $\mathbb{R}$-trees and 
$\alpha$-stable trees and collect tools required to obtain our results, namely the rigorous setup of RDEs, P\'{o}lya urn models, the Chinese 
restaurant process and Marchal's algorithm. Section \ref{Recursive} is dedicated to establishing an RDE for the law  of the $\alpha$-stable tree 
and indicating other fixpoint solutions to the same RDE. In Section \ref{secuniq}, we obtain the uniqueness and attraction properties of the RDE 
solution up to multiplicative constants. The latter arguments are in a general setup where $g$ is the single-point concatenation operation, but the 
distribution of $\xi$ is just subject to some non-degeneracy assumptions.

\section{Preliminaries} \label{Preliminaries}

We introduce several background formalisms and theories on metric spaces of $\mathbb{R}$-trees, $\alpha$-stable trees, urn schemes and recursive
distribution equations. We also state a general lemma that we will use to establish independence.%\pagebreak

\subsection{$\mathbb{R}$-trees and topologies on sets of (weighted or marked) $\mathbb{R}$-trees} \label{R-trees subection}

\begin{definition}[$\mathbb R$-tree] \label{R-tree} \rm
A metric space $ \left( \mathcal{T},d \right)$ is an $\mathbb{R}$\textit{-tree} if for every $a,b\in \mathcal{T}$, the following two conditions hold.
\begin{itemize}

\item[(i)] There exists a unique isometry $f_{a,b}\colon\left[0,d(a,b)\right] \rightarrow \mathcal{T}$ such that 
$f_{a,b}(0)=a$ and $f_{a,b}(d(a,b))=b$. In this case, let $\llbracket a,b\rrbracket$ denote the image $f_{a,b}\left(\left[0,d(a,b)\right]\right)$.

\item[(ii)] If $h\colon\left[0,1\right] \rightarrow \mathcal{T}$ is a continuous injective map with $h\left(0\right) = a$ and $h\left(1\right) = b$, then $h\left(\left[0,1\right]\right) = \llbracket a,b \rrbracket$, 
i.e. the only non self-intersecting path from $a$ to $b$ is $\llbracket a,b \rrbracket$.
\end{itemize}
\end{definition}
A \textit{rooted} $\mathbb{R}$-tree $(\mathcal{T},d,\rho)$ is an $\mathbb{R}$-tree $(\mathcal{T},d)$ with a distinguished vertex $\rho\in \mathcal{T}$ called the \textit{root}. The \textit{degree} of a vertex $a \in \mathcal{T}$ is the number of connected components of  $\mathcal{T}\setminus \lbrace a \rbrace$. A \textit{leaf} is a vertex $a \in \mathcal{T} \setminus \lbrace \rho \rbrace$ with degree one. We denote the set of leaves in $\mathcal{T}$ by $\mathcal{L}\left(\mathcal{T}\right)$. We say that $ a \in \mathcal{T} \setminus \lbrace \rho \rbrace $ is a \textit{branch point} if its degree is at least three. Finally, for any $a \in \mathcal{T}$, we define the \textit{height} of $a$ as $d\left(\rho,a\right)$, and the \textit{height} of $\mathcal{T}$ as $\textnormal{ht}\left(\mathcal{T}\right):=\sup_{a \in \mathcal{T}} d\left(\rho,a\right)$.
 
Two rooted $\mathbb{R}$-trees $\left(\mathcal{T},d,\rho\right)$ and $\left(\mathcal{T}',d',\rho'\right)$ are \textit{{\rm GH}-equivalent} if there exists an isometry $f\colon \mathcal{T} \rightarrow \mathcal{T}'$ such that $f\left(\rho\right)=\rho'$. The set of GH-equivalence classes of compact rooted $\mathbb{R}$-trees is denoted by $\mathbb{T}$.  
The \textit{Gromov--Hausdorff distance} between two rooted compact $\mathbb{R}$-trees $\left(\mathcal{T},d,\rho\right)$ and $\left(\mathcal{T}',d',\rho'\right)$ is defined as
\begin{equation} \label{GH metric}
d_{\rm GH}((\mathcal{T},d,\rho),(\mathcal{T}',d',\rho')):=\inf_{\phi,\phi'}\left(\delta_{\rm H}(\phi(\mathcal{T}),\phi'(\mathcal{T}')) \vee \delta (\phi(\rho),\phi'(\rho'))\right)
\end{equation} 
where the infimum is taken over all metric spaces $\left(X,\delta\right)$ and all isometric embeddings $\phi\colon\mathcal{T} \rightarrow X$ and $\phi'\colon\mathcal{T}' \rightarrow X$, and where $\delta_{\rm H}$ is the Hausdorff metric on compact subsets of $(X, \delta)$. The Gromov--Hausdorff distance only depends on the GH-equivalence classes of $\left(\mathcal{T},d,\rho\right)$ and $\left(\mathcal{T}',d',\rho'\right)$ and induces a metric on $\mathbb{T}$, which we also denote by $d_{\rm GH}$. 

There is an alternative characterisation of the Gromov--Hausdorff metric \cite[Theorem~7.3.25]{burago2001course}. Given two compact metric spaces $(X,\delta)$ and $(X',\delta')$, a \textit{correspondence} between $X$ and $X'$ is a subset $\mathcal{R} \subseteq X \times X'$ such that for every $x \in X$, there exists at least one $x' \in X'$ such that $(x,x')\in \mathcal{R}$, and conversely, for every $y' \in X'$, there exists at least one $y \in X$ such that $(y,y')\in \mathcal{R}$. The \textit{distortion} of this correspondence $\mathcal{R}$ is defined as
\begin{equation} \label{Distortion}
\textnormal{dis}(\mathcal{R}):=\sup\big\lbrace\left| \delta(x,y)-\delta'(x',y')\right|\colon(x,x'),(y,y') \in \mathcal{R}\big\rbrace.
\end{equation}
In our setting of two compact rooted $\mathbb{R}$-trees $\left(\mathcal{T},d,\rho\right)$ and $\left(\mathcal{T}',d',\rho'\right)$, we obtain
\begin{equation}
d_{\rm GH}\left(\left(\mathcal{T},d,\rho\right),\left(\mathcal{T}',d',\rho'\right)\right)= \frac{1}{2} \displaystyle\inf_{\mathcal{R}\in\mathcal{C}\left(\mathcal{T},\mathcal{T}'\right)}\text{dis}\left(\mathcal{R}\right), \label{distortion}
\end{equation} 
where $\mathcal{C}(\mathcal{T},\mathcal{T}')$ is the set of all correspondences $\mathcal{R}$ between $(\mathcal{T},d,\rho)$ and $(\mathcal{T}',d',\rho')$ which have $(\rho,\rho')$ in correspondence, i.e.\ for which $(\rho,\rho^\prime)\in\mathcal{R}$. %\pagebreak

We will want to specify a marked point on a compact rooted $\mathbb{R}$-tree. %We will only consider compact rooted $\mathbb{R}$-trees marked at a single point $x$. % and denote this space by $\mathcal T^{x} := \left(\mathcal{T},d,\rho,x\right)$. 
We refer the reader to \cite[Section~6.4]{miermont2007tessellations} for further extensions. Given two marked compact rooted $\mathbb{R}$-trees $\left(\mathcal{T},d,\rho,x\right)$ and $\left(\mathcal{T'},d',\rho',x'\right)$, the \textit{marked Gromov--Hausdorff distance} is defined as
\begin{equation*} %\label{GH marked metric}
d_{\rm GH}^{\rm m}\!\left(\!\left(\mathcal{T}\!,d,\rho,x\right)\!,\left(\mathcal{T'}\!,d'\!,\rho'\!,x'\right)\!\right) := \inf_{\phi,\phi'}\!\left(\delta_{\rm H}\!\left(\phi\!\left(\mathcal{T}\right)\!,\phi'\!\left(\mathcal{T}'\right)\right)\vee \delta \!\left(\phi\!\left(\rho\right)\!,\phi'\!\left(\rho'\right)\right)\vee \delta\!\left(\phi\!\left(x\right),\phi'\!(x')\right)\right),
\end{equation*} 
where the infimum is taken over all metric spaces $\left(X,\delta\right)$ and all isometric embeddings $\phi\colon\mathcal{T} \rightarrow X$ and $\phi'\colon\mathcal{T}' \rightarrow X$. We say two marked compact rooted $\mathbb{R}$-trees are ${\rm GH}^{\rm m}$-\textit{equivalent} if there exists an isometry $f \colon \mathcal{T} \rightarrow \mathcal{T}'$ such that $f(\rho) = \rho'$ and $f(x) = x'$. We denote the set of equivalence classes of marked compact rooted $\mathbb{R}$-trees by $\mathbb{T}_{\rm m}$. The marked Gromov--Hausdorff distance only depends on the ${\rm GH}^{\rm m}$-equivalence classes of $\left(\mathcal{T},d,\rho,x\right)$ and induces a metric on $\mathbb{T}_{\rm m}$, which we also denote by $d^{\rm m}_{\rm GH}$. 
In the spirit of \eqref{distortion}, we obtain for marked compact rooted $\mathbb{R}$-trees $(\mathcal{T},d,\rho,x)$ and $(\mathcal{T}',d',\rho',x')$, 
\begin{equation} \label{marked distortion}
d_{\rm GH}^{\rm m}\left(\left(\mathcal{T},d,\rho,x\right),\left(\mathcal{T}',d',\rho',x'\right)\right)= \frac{1}{2} \displaystyle\inf_{\mathcal{R}\in\mathcal{C}^{\rm m}\left(\mathcal{T},\mathcal{T}'\right)}\text{dis}\left(\mathcal{R}\right).\pagebreak
\end{equation}
where we denote by $\mathcal{C}^{\rm m}(\mathcal{T},\mathcal{T}')$ the set of all correspondences between $(\mathcal{T},d,\rho,x)$ and $(\mathcal{T}',d',\rho',x')$ which have $(\rho,\rho')$ and $(x,x')$ in correspondence; 
cf. \cite[Proposition~9(i)]{miermont2007tessellations}.

Suppose now that $(X,\delta)$ is a complete metric space. Then $\left(X,\delta,\mu\right)$ is a \textit{metric measure space} if $(X, \delta)$ is further equipped with a Borel probability measure $\mu$. % with respect to $\delta$ on $X$. 
We define a \textit{weighted} $\mathbb{R}$-tree as a compact rooted $\mathbb{R}$-tree $\left(\mathcal{T},d,\rho\right)$ equipped with a Borel probability measure $\mu$, which we refer to as \em mass measure\em. We will often write $\mathcal T$ for a weighted $\mathbb R$-tree, the distance, the root and the mass measure being implicit. 
For two weighted $\mathbb{R}$-trees $\left(\mathcal{T},d,\rho,\mu\right)$, $\left(\mathcal{T}',d',\rho',\mu'\right)$, the Gromov--Hausdorff--Prokhorov distance is defined as
\begin{equation*} %\label{GHP metric}
d_{\rm GHP}((\mathcal{T}\!,d,\rho,\mu),(\mathcal{T}'\!,d'\!,\rho'\!,\mu')):=\inf_{\phi,\phi'} \left( \delta_{\rm H}(\phi(\mathcal{T}),\phi'(\mathcal{T}'))\vee\delta(\phi(\rho),\phi'(\rho'))\vee\delta_{\rm P}(\phi_{*}\mu,\phi_{*}'\mu'))\right)\!,
\end{equation*}
where the infimum is taken over all metric spaces $\left(X,\delta\right)$ and all isometric embeddings $\phi\colon\mathcal{T} \rightarrow X$ and $\phi'\colon\mathcal{T}' \rightarrow X$, $\delta_{\rm P}$ denotes the Prokhorov-metric, and $\phi_{*}\mu$, $\phi_{*}'\mu'$ are the push-forwards of $\mu, \mu'$ under $\phi,\phi'$ respectively. 

Two weighted $\mathbb{R}$-trees $\left(\mathcal{T},d,\rho,\mu\right)$ and $\left(\mathcal{T}',d',\rho',\mu'\right)$ are considered \textit{{\rm GHP}-equivalent} if there is an isometry $f\colon\left(\mathcal{T},d,\rho,\mu\right) \rightarrow \left(\mathcal{T}',d',\rho',\mu'\right)$ such that $f(\rho)=\rho'$ and $\mu'$ is the push-forward of $\mu$ under $f$. Denote the set of equivalence classes of weighted $\mathbb{R}$-trees by $\mathbb{T}_{\rm w}$. The Gromov--Hausdorff--Prokhorov distance naturally induces a metric on $\mathbb{T}_{\rm w}$.

\begin{proposition} The spaces $\left(\mathbb{T},d_{\rm GH}\right)$, $\left(\mathbb{T}_{\rm m},d_{\rm GH}^{\rm m}\right)$ and $\left(\mathbb{T}_{\rm w},d_{\rm GHP}\right)$ are Polish.
\end{proposition}

\begin{proof}
See, e.g., \cite[Theorem~4.23]{e08}, \cite[Proposition~9(ii)]{miermont2007tessellations} and \cite[Theorem~2.7]{adh13}.
\end{proof}

In \cite{a91a,a91b,a93}, Aldous originally built his theory of %a special class of weighted $\mathbb{R}$-trees with infinitesimally short edges, 
%known as 
\textit{continuum trees}, %under an embedding into the 
in $\ell_1(\mathbb{N})$. %, endowing compact subsets with the Hausdorff metric
%and probability measures with the topology of weak convergence (induced by the Prokhorov metric). 
Indeed, some of our arguments will benefit from
specific representatives in $\ell_1(\mathbb{U})$, where $\mathbb{U}$ is the countable set of integer words. In any case, 
\cite[Theorem~3]{a93} connects Aldous's $\ell_1(\mathbb{N})$ embedding and the above setup of weighted $\mathbb{R}$-trees. So, we make the 
following definition.

\begin{definition}[Continuum Random Tree] \label{Continuum Tree}\rm
A weighted $\mathbb{R}$-tree $\left(\mathcal{T},d,\rho,\mu\right)$ is a \textit{continuum tree} if the Borel probability measure $\mu$ satisfies the following properties.
\begin{enumerate}
\item[(i)] $\mu\left(\mathcal{L}\left(\mathcal{T}\right)\right) = 1$, that is, $\mu$ is supported by the leaves of $\mathcal{T}$.
\item[(ii)] $\mu$ is non-atomic, that is, if $ a \in \mathcal{L}\left(\mathcal{T}\right) $, then $ \mu\left(\lbrace a \rbrace \right) = 0$.
\item[(iii)] For every $ a \in \mathcal{T} \setminus \mathcal{L}\left(\mathcal{T}\right) $, we have $\mu\left(\mathcal{T}(a)\right) > 0$, where $\mathcal{T}(a) := \lbrace \sigma \in \mathcal{T}\colon a \in \llbracket \rho, \sigma \rrbracket \rbrace $ is the subtree above $a$ in $\mathcal{T}$.
\end{enumerate}
A \textit{Continuum Random Tree} (CRT) is a random variable valued in a space (of GHP-equiva\-lence classes) of continuum trees. 
\end{definition}
Note that conditions (i) and (ii) above imply that a continuum tree has uncountably many leaves. It is not obvious how to determine the distribution of a CRT simply by its definition. To do this, it is useful to have a notion of \textit{reduced trees}.

\begin{definition}[Reduced tree] \label{Reduced subtrees}\rm
Let $\left(\mathcal{T},d,\rho,\mu\right)$ be a CRT and $ m \geq 1$. A \textit{uniform sample} of $m$ points according to the measure $\mu$ is a vector $(V_1, \ldots, V_m)$ such that $V_i \sim \mu$, $i=1,\ldots,m$, are i.i.d..
The associated \textit{$m$-th reduced subtree} of $\left(\mathcal{T},d,\rho,\mu\right)$ is the subtree of $\mathcal{T}$ spanned by $V_1,\ldots,V_m$ and $\rho$, i.e.\ $\bigcup_{1\le j\le m}\llbracket\rho,V_j\rrbracket$. 
\end{definition}

The distribution of the $m$-th reduced subtree is fully specified by its \textit{tree shape} when regarded as a discrete, graph-theoretic, rooted 
tree with $m$ labelled leaves, and by its \textit{edge lengths}. The consistent system of $m$-th reduced subtree distributions, $m\ge 1$, 
may be regarded as a system of finite-dimensional distributions of a CRT \cite{ag15}. It is well-known that they 
determine the distribution of a CRT on $\bT_{\rm w}$.

We now turn to Marchal's algorithm which leads to the definition of a special class of continuum random trees, the $\alpha$-stable trees with parameter $\alpha \in (1,2]$.

\subsection{Marchal's algorithm and $\alpha$-stable trees} 
\label{Marchal subsection}

Marchal's random growth algorithm generalises R\'{e}my's algorithm \cite{remy1985procede}, and also relates to Marchal's earlier work on the Lukasiewicz correspondence of random trees to excursions of a simple random walk converging to a Brownian excursion \cite{marchal2003constructing}. We adapt the notation employed in Curien and Haas \cite{curien2013stable} in the following. 

\begin{definition}[Marchal's algorithm] \label{Marchal's algorithm}\rm
 Given a parameter $\alpha\in(1,2]$, we recursively construct a sequence $\left(\mathbf{T}_{\alpha}(n)\right)_{n \geq 1}$ valued in the set of leaf-labelled discrete trees, with $\mathbf{T}_{\alpha}(n)$ having $n$ leaves and a root, as follows. 
\begin{itemize}
\item[(I)] Initialise $\mathbf{T}_{\alpha}(1)$ as the unique tree with one edge and two labelled endpoints, $A_0$ and $A_1$. Regard $A_0$ as the root and $A_1$ as a marked leaf. 
\item[(II)] For $n \geq 1$, given $\mathbf{T}_{\alpha}(n)$, assign weight $\alpha - 1$ to each edge of $\mathbf{T}_{\alpha}(n)$,  weight $d - 1 -\alpha$ to each branch point of degree $d \geq 3$, and no weight to other vertices. Choose an edge or a branch point of $\mathbf{T}_{\alpha}(n)$ with probability proportional to its weight. 
\item[(III)] Distinguish two cases depending on the selection in (II).\vspace{-0.2cm}
\begin{itemize}
\item[(a)] If an edge was selected, split the chosen edge into two edges at its midpoint by a new middle vertex denoted by $V_{n+1}$. At $V_{n+1}$, attach a new edge carrying the $(n+1)$-st leaf, denoted by $A_{n+1}$.\vspace{0.1cm}
\item[(b)] If a branch point was selected, attach a new edge carrying the $(n+1)$-st leaf at the chosen vertex. Denote the new leaf by $A_{n+1}$.
\vspace{-0.1cm}
\end{itemize}
\item[(IV)] Repeat from (II) with $n \mapsto n + 1$.%\pagebreak
\end{itemize}
Set $\widetilde{I}:=\lbrace k\ge 2\colon V_k \textnormal{ is created} \rbrace$, and define the limiting set of vertices at time $\infty$ as
$$ \mathbf{T}_{\alpha}(\infty) := \bigcup_{n \geq 0} \lbrace A_n \rbrace \cup \bigcup_{k \in \widetilde{I}} \lbrace V_k \rbrace. $$
\end{definition}

Define the measure $W(\cdot)$ which assigns the total weight to sub-structures in Marchal's algorithm. It is easy to see that, regardless of tree shape, for all $n \geq 1$, the total weight of the tree is $W(\mathbf{T}_{\alpha}(n))=n\alpha-1$. The distribution of the shape of the trees constructed in Marchal's algorithm was given in \cite[Theorem~1]{marchal2008note}:

%We have the following result on the tree shapes in Marchal's algorithm from \cite[Theorem~1]{marchal2008note}. 

\begin{proposition} \label{Marchal tree shape prob}
Suppose $\mathbf{t}$ is a given leaf-labelled tree with $n$ leaves and a root, where $n \geq 2$, then the tree shape of $\mathbf{T}_{\alpha}(n)$ has distribution 
$$\mathbb{P}\left(\mathbf{T}_{\alpha}(n)=\mathbf{t}\right) = \frac{\prod_{v \in \mathbf{t}}p_{\textnormal{deg}(v)}}{\prod^{n-1}_{i=1}(i\alpha-1)},$$
where $p_1 = 1$, $p_2 = 0$, and $p_k=\left|\prod^{k-2}_{i=1}(\alpha-i)\right|$ for $k \geq 3$. 
\end{proposition}

In the limit, a subtlety of Marchal's algorithm is that, almost surely, no two vertices chosen from $\mathbf{T}_{\alpha}(\infty)$ are adjacent. Suppose $u$ and $v$ are two vertices incident to edge $e$ at time $n_0$, then almost surely, we observe (countably) infinitely many branch points added into the path between the end vertices of $e$, as Marchal's algorithm progresses.

To turn the limiting object into an $\mathbb{R}$-tree, we take the natural \textit{completion} of $\mathbf{T}_{\alpha}(\infty)$ by `filling in-between' the countably many pairwise non-adjacent vertices. More precisely, between two chosen points $u,v \in \mathbf{T}_{\alpha}(n)$, the above entails that the graph distance between them tends to infinity as $n \rightarrow \infty$. By rescaling this distance appropriately and by identifying a suitable $L^2$-bounded martingale, invoking the Martingale Convergence Theorem, Marchal demonstrates the following limiting behaviour \cite[Theorem~2]{marchal2008note}.

\begin{proposition} \label{Marchal T_inf metric} For $\alpha\!\in\!(1,2]$, let $\beta\!:=\! 1\!-\! 1/\alpha\in(0,1/2]$. For all $u,v \in \mathbf{T}_{\alpha}(\infty)$, the limit $$ d(u,v) = \lim_{n \rightarrow \infty} {n^{-\beta}} d_n(u,v)$$ exists a.s., where $d_n$ is the graph distance on $\mathbf{T}_{\alpha}(n)$. Furthermore, the completion \linebreak $\left(\overline{\mathbf{T}_{\alpha}(\infty)},d\right)$ of $(\mathbf{T}_\alpha(\infty),d)$ is an $\mathbb{R}$-tree. 
\end{proposition}
We may regard $\overline{\mathbf{T}_{\alpha}(\infty)}$ as the \textit{scaling limit} of Marchal's algorithm as an $\mathbb{R}$-tree. Combining these observations, if $\left(T_{\alpha}(n)\right)_{n \geq 1}$ is an $\bR$-tree representation of $\left(\mathbf{T}_{\alpha}(n)\right)_{n \geq 1}$, then the limit
\begin{equation} \label{Marchal algorithm convergence}
\frac{T_{\alpha}(n)}{\alpha n^{\beta}} \rightarrow \mathcal{T}_{\alpha} \quad \text{ as } n \rightarrow \infty
\end{equation}
holds as a convergence of finite-dimensional distributions of reduced subtrees for some random $\mathbb R$-tree $\mathcal T_\alpha$. \cite[Corollary~24]{haas2008continuum} checks that $\mathcal{T}_{\alpha}$ may be constructed on the same probability space supporting $\left(\mathbf{T}_{\alpha}(n)\right)_{n \geq 1}$ with \eqref{Marchal algorithm convergence} holding in probability in the Gromov--Hausdorff sense. We state an improved result by Curien and Haas \cite[Theorem~5(iii)]{curien2013stable}. 
\begin{proposition} \label{Marchal GHP Curien and Haas Theorem}
 Let $\mu_n$ denote the empirical mass measure on the leaves of $T_{\alpha}(n)$, let $d_n$ be the graph distance on $T_{\alpha}(n)$, and let
 $\rho_n$ be the root. Then
 $$ \left(T_{\alpha}(n),\frac{d_n}{\alpha n^{\beta}},\rho_n,\mu_n\right) \overset{a.s.}{\longrightarrow} 
    \left(\mathcal{T}_{\alpha},d_{\alpha},\rho_\alpha,\mu_{\alpha}\right) \quad \text{ as } n \rightarrow \infty,$$
 in the Gromov--Hausdorff--Prokhorov topology, for some CRT $(\mathcal T_\alpha, d_\alpha, \rho_\alpha, \mu_\alpha)$. 
\end{proposition}

\begin{definition}[$\alpha$-stable tree]\rm We call $(\mathcal T_\alpha, d_\alpha, \rho_\alpha, \mu_\alpha)$ the \textit{$\alpha$-stable tree}, $\alpha \in (1,2]$.
\end{definition}

It is often useful to parametrize the $\alpha$-stable tree by an index $\beta:=1-1/\alpha \in (0,1/2]$, as in Proposition \ref{Marchal T_inf metric}. 
We often rescale trees: distances by $c^\beta$ and masses by $c$, as in 
$$\left(\mathcal T_\alpha, c^\beta d_\alpha, \rho_\alpha, c \mu_\alpha\right).$$

When $\alpha=2$, no weight is ever given to a vertex of $\mathbf{T}_{2}(n)$, $n\geq 1$, in the second step of Marchal's algorithm. In the scaling limit, this coheres with the fact that $\mathcal{T}_2$ is binary a.s.. 

Note that the tree $(\mathcal T_\alpha,d_\alpha,\rho_\alpha,\mu_\alpha)$ induces a distribution $\varsigma_\alpha$ on $\mathbb{T}_{\rm w}$. We call 
the distribution $\varsigma_\alpha$ the \textit{law of the $\alpha$-stable tree}. Similarly, we will consider the distribution 
$\varsigma_\alpha^{\rm m}$ of $(\mathcal T_\alpha, d_\alpha, \rho_\alpha, x_\alpha)$ on $\mathbb{T}_{\rm m}$ when $x_\alpha \sim \mu_\alpha$ is a 
marked element of $\mathcal T_\alpha$ sampled from $\mu_\alpha$, which we call the \textit{law of the marked $\alpha$-stable tree}.

At this juncture, it is instructive to introduce further developments in analogous constructions of $\alpha$-stable trees, and more general trees, based on Marchal's algorithm. 

Marchal's algorithm is a special case of Chen, Ford and Winkel's \textit{alpha-gamma model} \cite{chenfordwinkel2009new}. The alpha-gamma model allows further discrimination between edges adjacent to a leaf (external edges) and the remaining internal edges. 

The distribution of the sequence of tree shapes obtained in the line-breaking construction of the stable tree introduced by Goldschmidt and Haas is the same as that obtained by Marchal's algorithm \cite[Proposition~3.7]{goldschmidt2015line}. However, Goldschmidt and Haas' constructions focus on distributions of \textit{edge lengths} rather than \textit{mass} in an $\alpha$-stable tree. 

Recently, Rembart and Winkel introduced a \textit{two-colour line-breaking construction} \cite[Algorithm~1.3]{rembart2016binary} unifying aspects of the alpha-gamma model, and Goldschmidt and Haas' line-breaking construction. It ascribes a notion of length to the weights at branch points of Goldschmidt and Haas' line-breaking algorithm by growing trees at these branch points.

Little emphasis has been placed on the recursive nature of Marchal's algorithm \textit{per se}. In \cite{curien2013stable}, Curien and Haas exploit this property to demonstrate a pruning procedure to obtain a rescaled $\alpha'$-stable tree from an $\alpha$-stable tree, where $1<\alpha<\alpha'\leq2$. They identified sub-constructions within Marchal's algorithm with parameter $\alpha$ that evolve as a time-changed Marchal algorithm with parameter $\alpha'$. We use a similar approach in Section \ref{Main result subsection} to find a recursive
distribution equation where the law of the $\alpha$-stable tree is a solution.

\subsection{P\'olya urns and Chinese restaurant processes} \label{GPU and CRP subsection}

We briefly recap the concepts of P\'olya urns and Chinese restaurant processes. 

Given $\beta > 0$ and $\theta > -\beta$, a random variable $L$ valued in $\left[0,\infty\right)$ has a \textit{generalised Mittag--Leffler distribution} with parameters $\left(\beta,\theta\right)$, denoted by $L \sim \text{ML}\left(\beta,\theta\right)$, if it has $p$-th moment\vspace{-0.1cm}
\begin{equation} \label{Mittag--Leffler distribution}
\mathbb{E}\left[L^p\right]=\frac{\Gamma(\theta+1)\Gamma\left(\theta/\beta+1+p\right)}{\Gamma(\theta/\beta+1)\Gamma\left(\theta+\beta p+1\right)}, \qquad p \geq 1.\vspace{-0.1cm}
\end{equation}

The Mittag--Leffler distribution is uniquely characterised by the moments \eqref{Mittag--Leffler distribution}, see e.g.\ \cite{pitman2006combinatorial}. 
It was shown in \cite[Lemma~11]{addario2016inverting} that $\alpha$ times the distance between two uniformly sampled points on an $\alpha$-stable tree has a $\textnormal{ML}(\beta, \beta)$ distribution, where $\beta = 1-1/\alpha$. As the $\alpha$-stable tree remains invariant under uniform re-rooting \cite[Theorem~11]{haas2009spinal}, this is the distribution of $\alpha$ times the distance between the root and a uniformly sampled point.

To analyse Marchal's algorithm, we will also use the following well-known \textit{aggregation property} of the Dirichlet distribution. 
\begin{proposition} \label{Dirichlet distn properties prop}
For $n\geq 2$, let $\beta_1,\ldots,\beta_n > 0$ and $Y := (Y_1,\ldots,Y_n) \sim \textnormal{Dir}\left(\beta_1,\ldots,\beta_n\right)$. Let $1 \leq m \leq n-1$. Then $Y':=\left(\sum_{i=1}^{m}Y_i,Y_{m+1},\ldots,Y_n\right)\sim \textnormal{Dir}\left(\sum_{i=1}^{m}\beta_i,\beta_{m+1},\ldots,\beta_n\right)$.
\end{proposition}

Dirichlet and Mittag--Leffler distributions arise naturally in a variety of urn models, see \cite{pitman2006combinatorial} and \cite{janson2006limit} respectively. For our purposes, we restrict attention to the following specification of P\'{o}lya's urn model. 

\begin{definition}[Generalised P\'olya urn] \label{Generalised Polya urn}\rm

Given $K\!\geq\!2$ and $\vec{\gamma} = (\gamma_1,\gamma_2,\ldots,\gamma_K)$ with $\gamma_1,\gamma_2,\ldots,\gamma_K > 0$, consider a P\'{o}lya urn scheme with $K$ colours, initialisation $\vec{\gamma}$ and step-size $t>0$ evolving in discrete time. Represent the $K$ colours by the set $C := \lbrace 1,2,\ldots, K\rbrace $.  We say that a random variable $X$ valued in $C$ has distribution $\kappa_{\vec{\gamma}}$  if $\mathbb{P}\left(X=j\right)={\gamma_j}({\sum^{K}_{i=1} \gamma_i})^{-1}$ for all $j \in C$. Generate a sequence of draws $\left(X_1,X_2,\ldots\right)$ from $C$ according to the following scheme:\vspace{-0.1cm}
\begin{itemize}
\item[(I)] Set $\vec{\gamma}_1 := \vec{\gamma}$, sample $X_1$ from $\kappa_{\vec{\gamma}_1}$.\vspace{-0.1cm} 
\item[(II)] For $ n \geq 1$, set $\vec{\gamma}_{n+1} := \vec{\gamma}_{n} + t\vec{e}_{X_n}$, where $\vec{e}_j$ denotes the $j$-th standard Euclidean basis vector of $\mathbb{R}^K$. Given $X_1, \ldots, X_n$, sample $X_{n+1}$ from $\kappa_{\vec{\gamma}_{n+1}}$.%\pagebreak 
\end{itemize}
For $j \in C$, denote the number of $j$-th coloured balls observed after $n$ draws by $$D^{(n)}_j := \displaystyle\sum^n_{i=1}\mathbf{1}( X_i=j),$$ and define the vector of relative frequencies of colours observed in the first $n$ draws as
$$ \left(P^{(n)}_1,P^{(n)}_2,\ldots,P^{(n)}_K \right) := \left(\frac{D^{(n)}_1}{n},\frac{D^{(n)}_2}{n},\ldots,\frac{D^{(n)}_K}{n}\right).$$
\end{definition} 
The relative frequencies of colours observed are known to converge to an almost sure limit, due to Blackwell and MacQueen \cite{blackwell1973ferguson}. 
\begin{proposition} \label{Polya Urn Theorem}
Given a sequence of draws $\left(X_1,X_2,\ldots\right)$ from the urn scheme in Definition \ref{Generalised Polya urn}, the relative frequencies in the draws satisfy  
\[ \left(P^{(n)}_1,P^{(n)}_2,\ldots,P^{(n)}_K \right) \overset{a.s.}{\longrightarrow} \left(P_1,P_2,\ldots,P_K\right) \quad \text{ as } n \rightarrow \infty, \]
where $\left(P_1,P_2,\ldots,P_K\right) \sim \textnormal{Dir}\left({\gamma_1}/{\alpha},{\gamma_2}/{\alpha},\ldots,{\gamma_K}/{\alpha}\right)$. 
Consequently, the proportions of colours in the urn satisfy  
\[ \left(\frac{\gamma_1+t D^{(n)}_1}{\sum_{i=1}^{K}\gamma_i+t n},\frac{\gamma_2+t D^{(n)}_2}{\sum_{i=1}^{K}\gamma_i+t n},\ldots,\frac{\gamma_K+t D^{(n)}_K}{\sum_{i=1}^{K}\gamma_i+t n} \right) \overset{a.s.}{\longrightarrow} \left(P_1,P_2,\ldots,P_K\right) \quad \text{ as } n \rightarrow \infty. \]
\end{proposition} 
A natural extension to the urn scheme introduced in Definition \ref{Generalised Polya urn} is the two-parameter \textit{Chinese restaurant process} (CRP), see Pitman \cite{pitman1996some}. 

\begin{definition}[Chinese restaurant process] \label{CRP Defn}\rm
Given $\beta\in[0,1]$ and $\theta>-\beta$, the two-parameter \textit{Chinese restaurant process} with a $\left(\beta,\theta\right)$ seating plan, denoted by $\textnormal{CRP}\left(\beta,\theta\right)$, proceeds as follows. Label customers by $n \geq 1$. Seat customer 1 at the first table. For $n \geq 1$, let $K_n$ denote the number of tables occupied after customer $n$ has been seated and let $N_j(n)$ denote the number of customers seated at the $j$-th table for $j \in \lbrace 1,\ldots,K_n \rbrace $. At the next arrival, conditional on $\left(N_1(n),\ldots,N_j(n)\right)$, customer $n + 1$
\begin{itemize}
\item sits at the $j$-th table with probability $\left(N_j(n)-\beta\right)/\left(n+\theta\right)$ for $j \in \lbrace 1,\ldots,K_n\rbrace $,
\item opens the $\left(K_n+1\right)$-st table with the complementary probability $\left(\theta+K_n\beta\right)/\left(n+\theta\right)$. %\pagebreak
\end{itemize}
For each $n \geq 1$, the process at step $n$ induces a partition  $\Pi_n := \left(\Pi_{n,1},\ldots,\Pi_{n,K_n}\right)$ of $\lbrace 1,\ldots,n \rbrace$ into blocks, given by the collection of customer labels at each occupied table, with blocks ordered by least labels. This induces a partition-valued process $\left(\Pi_n, n\geq1\right)$.
\end{definition}

As with P\'{o}lya urn schemes, the CRP also satisfies limit theorems associated with the Dirichlet and Mittag--Leffler distributions, cf. \cite[Theorem~3.2 and Theorem~3.8]{pitman2006combinatorial}.  

\begin{proposition} \label{CRP Limit Theorems}
Consider a Chinese restaurant process with parameters $\beta \in (0,1)$ and $\theta > -\beta$. Then the number of tables $K_n$ at time $n$ satisfies
$$ n^{-\beta}{K_n} \overset{a.s.}{\longrightarrow} K_{\infty} \quad \text{ as } n \rightarrow \infty,$$
where $K_{\infty} \sim \textnormal{ML}(\beta,\theta)$. 
Furthermore, relative table sizes have almost sure limits
$$ \left(\frac{N_1(n)}{n},\frac{N_2(n)}{n},\ldots,\frac{N_{K_n}(n)}{n},0,0,\ldots\right) \overset{a.s.}{\longrightarrow} \left(W_1,\overline{W}_1 W_2,\overline{W}_1\overline{W}_{2}W_{3},\ldots\right) \quad \text{ as } n \rightarrow \infty,$$
where $W_j \sim \textnormal{Beta}\left(1-\beta,\theta+ j\beta\right)$, $j\ge 1$, are independent and $\overline{W}_j := 1-W_j$ for all $j \geq 1 $. 
\end{proposition}

The distribution of the vector $(P_1,P_2,P_3,\ldots):=\left(W_1,\overline{W}_1 W_2,\overline{W}_1\overline{W}_{2}W_{3},\ldots\right)$ as defined in Proposition \ref{CRP Limit Theorems} is a \textit{Griffiths--Engen--McCloskey distribution} with parameters $(\beta,\theta)$, denoted by $\textnormal{GEM}(\beta,\theta)$. Ordering $(P_i, i \geq 1)$ in decreasing order yields a Poisson--Dirichlet distribution with parameters $(\beta,\theta)$, for short ${\rm PD}(\beta, \theta)$, i.e. $$\left(P^{\downarrow}_i, i \geq 1\right) := \left(P_i, i \geq 1\right)^{\downarrow} \sim {\rm PD}\left( \beta, \theta\right).$$

\subsection{Recursive distribution equations}

\label{RDE subsection}

Before we can introduce our specific recursive distribution equation (RDE) for the stable tree, it is instructive to review RDEs in their full generality, as presented in \cite[Section~2.1]{ab05}. Denote our underlying probability space by $\left(\Omega,\mathcal{F},\mathbb{P}\right)$. Given two measurable spaces $\left(\mathbb{S},\mathcal{F}_\mathbb{S}\right)$ and $\left(\Theta,\mathcal{F}_{\Theta}\right)$, construct the product space
\begin{equation} \label{RDE prodspace}
\Theta^{*} := \Theta \times \bigcup\limits_{0 \leq m \leq \infty} \mathbb{S}^m,
\end{equation}
where the union is disjoint over $\mathbb{S}^m$, the space of $\mathbb{S}$-valued sequences of lengths $0 \leq m \leq \infty$, and where $\mathbb{S}^{0}:=\lbrace \Delta \rbrace$ is the singleton set and $\mathbb{S}^{\infty}$ is constructed as a typical sequence space. 

Equip $\Theta^{*}$ with the product sigma-algebra. Let $g\colon \Theta^{*} \rightarrow \mathbb{S}$ be a measurable map, and define random variables $\left(\mathcal{S}_{i}, i \geq 0 \right) \in \mathbb{S}^{\infty}$,  $\left(\xi,N\right) \in \Theta \times \overline{\mathbb{N}}:= \Theta \times \lbrace 0,1,\ldots; \infty \rbrace $ as follows.
\begin{enumerate}
\item[(i)] $(\xi,N) \sim \nu$, where $\nu$ is a probability measure on $\Theta \times \overline{\mathbb{N}}$.
\item[(ii)] $\mathcal{S}_{i} \sim \eta$, $i \geq 0$, i.i.d., where $\eta$ is a probability measure on $\mathbb{S}$.
\item[(iii)] $(\xi,N)$ and $\left(\mathcal{S}_i, i \geq 0\right)$ are independent.\pagebreak
\end{enumerate}
Denote by $\mathcal{P}\left(\mathbb{S}\right)$ the set of probability measures on $\left(\mathbb{S},\mathcal{F}_\mathbb{S}\right)$. Given the distribution $\nu$ on $\Theta \times \overline{\mathbb{N}} $, we obtain a mapping
\begin{equation} \label{RDEmap} \Phi\colon \mathcal{P}(\mathbb{S}) \rightarrow \mathcal{P}(\mathbb{S}), \qquad \eta \mapsto \Phi(\eta), \end{equation}
where $\Phi(\eta)$ is the distribution of 
$\mathcal{S} := g(\xi,\mathcal{S}_{i}, 0 \leq i \leq^{*}\!N) \label{rtfequ}$, 
and where the notation $\leq^{*}\!N$ means $\leq N$ for $N < \infty$ and $< \infty$ for $N = \infty$. This lends itself to a fixpoint perspective of RDEs, where we wish to find a distribution of $\mathcal{S}$ such that 
\begin{equation} \label{RDEfixedpt}
\eta = \Phi(\eta) \iff \mathcal{S} \overset{d}{=} g(\xi,\mathcal{S}_{i}, 0 \leq i \leq^{*}\!N) \quad \text{ on } \mathbb{S}.
\end{equation}

In a \textit{recursive tree framework}, the approach of \eqref{rtfequ} is extended recursively to $\mathcal S_i, i \geq 1,$ and beyond. To this end, we will work with the Ulam--Harris-indexation 
$$\mathbb{U}:=\bigcup_{n \geq 0} \mathbb{N}^n,\qquad\mbox{where }\mathbb{N}:=\{0,1,2,\ldots\}. \label{Ulam} $$

Consider a sequence of i.i.d.\ $\Theta \times \overline{\mathbb N}$-valued random variables $(\xi_{\mathbf{u}}, N_{\mathbf{u}}), \mathbf{u} \in \mathbb{U}$. Furthermore, suppose that there are random variables $\tau_{\mathbf{u}}, {\mathbf{u}} \in \mathbb{U}$, possibly on an extended probability space, as follows.
\begin{itemize}
\item[(i)] For all $\mathbf{u} \in \mathbb{U}$, 
\begin{equation} \tau_{\mathbf{u}}=g\left(\xi_{\mathbf{u}}, \tau_{{\mathbf{u}}j}, 1 \leq j \leq^{*}\!N_{\mathbf{u}}\right) \quad \text{ a.s..} \label{45} \end{equation}
\item[(ii)] The variables $(\tau_{\mathbf{u}}, \mathbf{u} \in \mathbb{N}^n)$ are i.i.d.\ with some distribution $\eta_n$, $n \geq 1$.
\item[(iii)] The variables $(\tau_{\mathbf{u}}, \mathbf{u} \in \mathbb{N}^n)$ are independent of the variables $(\xi_{\mathbf{u}}, N_\mathbf{u}, \mathbf{u} \in \bigcup_{k=0}^n \mathbb{N}^k)$.%\pagebreak 
\end{itemize}

In this setup, we may define a \textit{recursive tree framework} as follows.

\begin{definition}[Recursive tree framework]\rm A pair $((\xi_{\mathbf{u}}, N_\mathbf{u}, \mathbf{u} \in \mathbb{U}), g)$ is called a \textit{recursive tree framework} if $(\xi_{\mathbf{u}}, N_\mathbf{u}, \mathbf{u} \in \mathbb{U})$ is an i.i.d.\ family of $\Theta\times \overline{\mathbb N}$-valued random variables $(\xi_{\mathbf{u}},  N_\mathbf{u}) \sim \nu, \mathbf{u} \in \mathbb{U}$, and $g\colon \Theta^* \rightarrow \mathbb{T}$ is a measurable map.
\end{definition}

If we enrich an RTF with the random variables $\tau_{\mathbf{u}}, \mathbf{u} \in \mathbb{U}$, we obtain a so-called \textit{recursive tree process} (RTP). Sometimes, RTPs are only considered up to generation $n$, that is, only for $\tau_{\mathbf{u}}, {\mathbf{u}} \in \bigcup_{k=0}^n \mathbb N^k$. We then speak of an \textit{RTP of depth $n$}. Such finite-depth RTPs can always be defined for any distribution $\eta_n$ of $\tau_{\mathbf{u}}, \mathbf{u} \in \mathbb{N}^n$, and $\eqref{45}$ for generations $n-1, \ldots, 0$. RTPs of infinite depth do not necessarily exist in general. We refer to \cite[Section 2.3]{ab05} for more details on RTFs and RTPs, and connections to Markov chains and Markov transition kernels. 

\subsection{An independence criterion}

To end the Preliminaries section, %Before we study limiting weight partitions in the subtrees of Marchal's algorithm and define our recursive distribution equation for the stable tree, 
we introduce an elementary lemma, which will help us verify certain required independences. We leave its proof to the reader.%\pagebreak% As we were not able to find a proof in the literature, we provide one here.

\begin{lemma} \label{Marchal independence technical Lemma}
Let $T$ be an a.s.\ finite stopping time with respect to a filtration $\left(\mathcal{F}_n\right)_{n \geq 1}$. Suppose that $X$ is a non-negative and bounded random variable satisfying, for each $n \geq 1$, 
$$ \mathbb{E}\left[X\mid\mathcal{F}_T\right] = \mathbb{E}\left[X\mid\mathcal{F}_m\right] \quad \text{a.s.},$$ 
for all $m\geq n$ on $\lbrace T=n \rbrace$. Then $\mathbb{E}[X|\mathcal{F}_T]=\mathbb{E}[X|\mathcal{F}_{\infty}]$ a.s.\ where $\mathcal F_\infty=\sigma\left( \mathcal F_n, n\geq 1 \right)$.
\end{lemma}

%\section{Recursive distribution equations for \texorpdfstring{$\mathbb{R}$}{R}-trees} 

\section{An RDE for $\mathbb{R}$-trees from Marchal's algorithm} \label{Main result subsection}\label{Recursive}
In this section, fix $\alpha \in (1,2]$ and let $\beta=1-1/\alpha \in (0,1/2]$. Unless ambiguity arises, we suppress $\alpha$ hereafter. Note that, in Marchal's algorithm, $\mathbf{T}(2)$ is deterministic, comprising a Y-shape with three leaves $A_0$, $A_1$ and $A_2$ and an internal vertex $V_2$. Denote the edges by $ e_0 := \llbracket A_0,V_2 \rrbracket$, $ e_1 := \llbracket A_1,V_2 \rrbracket$ and $ e_2 := \llbracket A_2,V_2 \rrbracket$. The following heuristic, implicitly employed in the proof of \cite[Proposition~10]{curien2013stable}, outlines the argument in this section.

The independent choice at each step of Marchal's algorithm entails that we have independent sub-constructions of Marchal's algorithm with parameter $\alpha$ evolving along each edge of $\mathbf{T}(2)$. This yields three independent copies of $\mathcal{T}_{\alpha}$, denoted by $\tau_0$, $\tau_1$ and $\tau_2$, subject to rescaling depending on the eventual proportion of mass distributed to each tree. For $\alpha \in (1,2)$, the internal vertex $V_2$ will give rise to a further countably infinite and independent collection of copies of $\mathcal{T}_{\alpha}$ a.s.. Denote this infinite collection by $\left(\tau_i, i \geq 3\right)$, which is independent of $\tau_0$, $\tau_1$ and $\tau_2$. We will rescale and concatenate our collection $\left(\tau_i, i \geq 0\right)$ of independent copies of $\mathcal{T}_{\alpha}$ at $V_2$ to get a copy of $\mathcal{T}_{\alpha}$. Denote the collection of scaling factors in the limit by $\xi = (\xi_i, i \geq 0)$ and the concatenation operator by $g$. We obtain an RDE $$\mathcal{T}_{\alpha} \overset{d}{=} g\left(\xi,\tau_i, i \geq 0\right)$$ in the form \eqref{RDEfixedpt}. To be rigorous, we need to address the following questions.
\begin{enumerate}
\item[1.] What is the distribution of the limiting scaling factors $\xi = (\xi_i, i \geq 0)$?
\item[2.] Are the random variables $\left(\tau_i, i \geq 0\right)$ independent of $\xi$, as well as of each other?
\item[3.] How do we construct the concatenation operation in a measurable way?  
\end{enumerate}

\subsection{The scaling factors $\xi=(\xi_i,i\ge 0)$}

For $i \in \lbrace 0,1,2 \rbrace$ and $n \geq 0$, define $\tau_i^{(n)}$ as the subtree of $\mathbf{T}(n+2)$ cut at $V_2$ containing the edge $e_i$. For example, we have $\tau_i^{(0)} = e_i$ for each $i \in \lbrace 0,1,2 \rbrace$. Let $\mathcal{K}_n$ denote the set of edges incident to $V_2$ in $\mathbf{T}(n+2)$ excluding $\lbrace e_i, i = 0,1,2 \rbrace$, and set $K_n = \left| \mathcal{K}_n \right|$. For $\mathcal{K}_n \neq \emptyset$, $\mathcal{K}_n = \lbrace e_j, j = 3,\ldots,K_n+2 \rbrace$, ordered according to least leaf labels. Define $\sigma^{(n)}$ as the remaining component of $\mathbf{T}(n+2)$ cut at $V_2$ excluding $\bigcup^2_{i=0} \tau_i^{(n)}$. If $\mathcal{K}_n = \emptyset$, then $\sigma^{(n)} = \emptyset$. Otherwise, $\sigma^{(n)} = \bigcup_{j=3}^{K_n+2} \tau_j^{(n)}$ is a union of subtrees $\lbrace \tau_j^{(n)}, j = 3,\ldots,K_n+2 \rbrace$ growing along their respective edges in $\mathcal{K}_n$. We illustrate this in Figure \ref{fig:dissertMarchal}. 

\begin{figure}[t]
  \includegraphics[width=\linewidth]{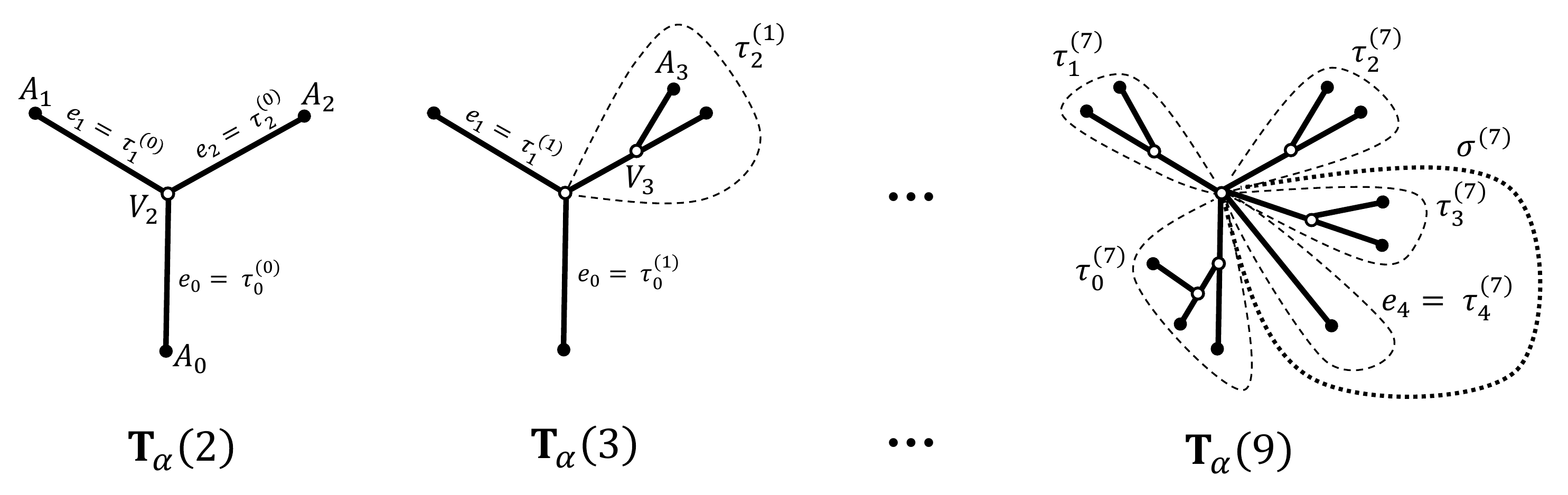}
  \caption{Illustration of Marchal's random growth algorithm and notation employed}
  \label{fig:dissertMarchal}
\end{figure}

Denote the number of leaves in $\tau_i^{(n)}$ excluding $V_2$ by $N_i(n)$ for all $i=0,1,\ldots,K_n+2$, and define its inverse $N^{-1}_i(n) := \inf \lbrace k \geq 0\colon N_i(k) = n \rbrace$ as the first time $k$ at which $\tau_i^{(k)}$ has $n$ leaves excluding $V_2$, with the convention $\inf \emptyset = \infty$. 

Regard $V_2$ as a (weightless) root from the perspective of each element of $\lbrace \tau_i^{(n)}, i = 1,\ldots,K_n+2 \rbrace$ and as a marked leaf of $\tau_0^{(n)}$. For each $i = 1,\ldots,K_n+2$, mark the first leaf created in $\tau_i^{(n)}$ by Marchal's algorithm, that is, the other endpoint of $e_i$ which is not $V_2$. Recall that $A_0$ is the root of $\tau_0^{(n)}$. 

In the limit, denote by $\tau_i^{(\infty)}$ the limiting set of vertices corresponding to the $i$-th subtree. Denote its associated $\mathbb{R}$-tree by $\tau_i$, obtained by the completion of $\tau_i^{(\infty)}$ in the scaling limit as described in Section \ref{Marchal subsection}. Likewise, define $\sigma^{(\infty)}$ and $\sigma$ as the limiting set of vertices and the $\mathbb{R}$-tree associated with $\sigma^{(n)}$, respectively. 

Recall that $W(\cdot)$ measures the total weight of a given sub-structure, e.g., for each $i \in \lbrace 0,1,2 \rbrace$, $W(\tau_i^{(0)})=\alpha-1.$ The following result shows that the weight of a particular subtree only depends on the number of leaves it has, and not on its shape. 

\begin{lemma} \label{Marchal total weight prop}
Regardless of its shape, the total weight of the $i$-th subtree is $W(\tau^{(n)}_i)\!=\!\alpha N_i(n)\!-\!1$ for $i = 0,1,\ldots,K_n+2$ and $n \geq 0$. 
\end{lemma}
\begin{proof} This follows simply by induction applied to each subtree. 
\end{proof}
%We are now able to give a result on the limiting weight partitions in the subtrees due to Marchal's algorithm. Recall that $\beta = 1 - 1/\alpha$. 
\begin{proposition} \label{Marchal mass distribution prop}
We have the following limiting results for weights. 
\begin{enumerate}
\item[(i)] For $\alpha = 2$, $K_n = 0$ a.s.\ for all $n \geq 0$. The relative weight split in $\mathbf{T}_{\alpha}(n)$ has an almost sure limit as $n \rightarrow \infty$ given by  
$$ \left(\frac{W\left(\tau_0^{(n)}\right)}{2n+3},\frac{W\left(\tau_1^{(n)}\right)}{2n+3},\frac{W\left(\tau_2^{(n)}\right)}{2n+3}\right) \overset{a.s.}{\longrightarrow} \left(X_0,X_1,X_2\right) $$
where $\left(X_0,X_1,X_2\right) \sim \textnormal{Dir}\left(1/2,1/2,1/2\right)$.  
\item[(ii)] For $\alpha \in (1,2)$, $K_n \rightarrow \infty$ as $n \rightarrow \infty$  almost surely. The relative weight split in $\mathbf{T}_{\alpha}(n)$ has an almost sure limit as $n \rightarrow \infty$ given by
\begin{equation}\left(\!\frac{W\left(\tau_0^{(n)}\right)}{(n\!+\!2)\alpha\!-\!1},\frac{W\left(\tau_1^{(n)}\right)}{(n\!+\!2)\alpha\!-\!1},\frac{W\left(\tau_2^{(n)}\right)}{(n\!+\!2)\alpha\!-\!1},\frac{W\!\left(\sigma^{(n)}\right) \!+\! W\!\left(\lbrace V_2 \rbrace\right)}{(n+2)\alpha-1}\!\right)\! \overset{a.s.}{\longrightarrow} \left(X_0,X_1,X_2,X_3\right), \label{weight limits} \end{equation}
where $\left(X_0,X_1,X_2,X_3\right) \sim \textnormal{Dir}\left(\beta,\beta,\beta,1-2\beta\right)$. Within the last part, denote the eventual proportion of weight distributed to the subtree $\tau_{i+2}$ by $P_i$ for $i \geq 1$. Then, $$\left(P_i, i \geq 1\right) \sim \textnormal{GEM}\left(1-\beta,1-2\beta\right).$$ In particular, the subtrees $\tau_i$, $i \geq 3$, have a relative weights partition that follows
a $\textnormal{PD}(1-\beta,1-2\beta)$ distribution, when ranked in decreasing order.
\end{enumerate} 
\end{proposition}
\begin{proof} We prove (ii). From Lemma \ref{Marchal total weight prop}, conditional on an edge or branch point in $\tau_i^{(n)}$ being selected in the next step of Marchal's algorithm, we increase the weight in $\tau_i^{(n)}$ by $\alpha$. It is easy to check that this also holds for $\sigma^{(n)}$ with one weighted copy of $V_2$ included. Hence,
\begin{equation} \label{Prop 4.2.2 eqn}
\left(W\left(\tau_0^{(n)}\right),W\left(\tau_1^{(n)}\right),W\left(\tau_2^{(n)}\right),W\left(\sigma^{(n)}\right) + W\left(\lbrace V_2 \rbrace\right)\right)
\end{equation}
evolves precisely as the P\'{o}lya urn scheme in Definition \ref{Generalised Polya urn} with $K\!=\!4$, initialisation vector $\vec{\gamma} = \left(\alpha\!-\!1,\alpha\!-\!1,\alpha\!-\!1,2\!-\!\alpha\right)$ and step-size $t\!=\!\alpha$. Therefore, \eqref{weight limits} holds by Proposition \ref{Polya Urn Theorem}. 

Next, we focus on the subtrees within $\sigma^{(n)}$. The above implies that $W\left(\sigma^{(n)}\right)+W\left(\lbrace V_2 \rbrace\right) \rightarrow \infty$ as $n \rightarrow \infty$ a.s.. So, a.s., we observe infinitely many leaves being added to $\left(\sigma^{(n)}, n\geq1\right)$. We may then condition on the times where a leaf is added to $\left(\sigma^{(n)}, n\geq1\right)$, say $(q_i, i \geq 1)$, where $1 \leq q_1 < q_2 < \cdots <q_n<q_{n+1}<\cdots$ is an infinite sequence a.s.. Conditional on the preceding event, the first leaf added creates $\tau_3^{(q_1)}$. At each $r \in \lbrace q_n,\ldots,q_{n+1}-1 \rbrace$, we have $n$ leaves (not including $V_2$) with $K_{q_n}$ subtrees whose union is $\sigma^{(r)}$. For $j = 3,\ldots,K_{q_n}+2$, $\tau^{(r)}_{j}$ has $N_j(q_n)$ leaves (not including $V_2$), and so has total weight $\alpha N_j(q_n) - 1$, by Lemma \ref{Marchal total weight prop}. Thus, as the total weight of $V_2$ is $2+K_{q_n}-\alpha$, the total weight of $\sigma^{(r)}$ and $\lbrace V_2 \rbrace$ is $\alpha n + (2-\alpha)$. At the next arrival time $q_{n+1}$, we add a leaf to $\tau^{(q_{n+1}-1)}_{j}$ with probability $\left(\alpha N_j(q_n)-1\right)/\left(\alpha n + 2-\alpha\right)$ and we create a new sub-tree with probability $\left(2 + K_{q_n} - \alpha\right)/\left(\alpha n + 2-\alpha\right)$. Regarding the leaves (excluding $V_2$) as customers and each subtree as a table, this models a Chinese restaurant process with parameters $(1-\beta,1-2\beta)$, according to Definition \ref{CRP Defn}. From Proposition \ref{CRP Limit Theorems}, $K_n \rightarrow \infty$ as $n \rightarrow \infty$  almost surely. Recall $q_1 < \infty$ almost surely, so we may assume $n \geq q_1$. From Proposition \ref{CRP Limit Theorems}, we can identify the almost sure limiting proportion of leaves split within subtrees of $\sigma$ as $\textnormal{GEM}(1-\beta,1-2\beta)$ holding along the increasing subsequence $(q_i, i \geq 1)$. That is,
$$ \left(\frac{N_3(q_n)}{n}, \ldots,\frac{N_{K_{q_n}+2}(q_n)}{n},0,0,\ldots\right) \overset{a.s.}{\longrightarrow} \left(P_i, i \geq 1\right) \quad \text{ as } n \rightarrow \infty, $$
where $\left(P_i, i \geq 1\right) \sim \textnormal{GEM}(1-\beta,1-2\beta)$. 
Write $N_{\sigma}(n)$ as the number of leaves in $\sigma^{(n)}$ excluding $V_2$. Noting that $N_{\sigma}(n) > 0$ for $n \geq q_1$, we may rephrase the above as
\begin{equation} \label{Marchal mass distrn prop EQN 1}
\left(\frac{N_3(n)}{N_{\sigma}(n)},\ldots,\frac{N_{K_{n}+2}(n)}{N_\sigma(n)},0,0,\ldots\right) \overset{a.s.}{\longrightarrow} \left(P_i, i \geq 1\right) \quad \text{ as } n \rightarrow \infty. 
\end{equation}
Using the relation $W\left(\sigma^{(n)}\right)+ W\left(\lbrace V_2 \rbrace \right) = \alpha N_{\sigma}(n) + 2 - \alpha$, and the aggregation property of the Dirichlet distribution in Proposition \ref{Dirichlet distn properties prop} applied to \eqref{Prop 4.2.2 eqn}, we get that 
\begin{equation} \label{Marchal mass distrn prop EQN 2}
\frac{N_{\sigma}(n)}{n} \overset{a.s.}{\longrightarrow} X_3 \quad \text{ as } n \rightarrow \infty,
\end{equation}
where $X_3 \sim \textnormal{Beta}(1-2\beta,3\beta)$. By the algebra of almost sure convergence,
$$ \left(\frac{N_3(n)}{n}, \ldots,\frac{N_{K_{n}+2}(n)}{n},0,0,\ldots\right) \overset{a.s.}{\longrightarrow} \left(X_3P_i, i \geq 1 \right) \quad \text{ as } n \rightarrow \infty. $$
Therefore, for all $j = 3,\ldots,K_n+2$ and $n \geq q_1$, as we have $W\left(\tau^{(n)}_{j}\right) = \alpha N_j(n) - 1$, the above implies that, jointly in $j$,
$$ \frac{W\left(\tau^{(n)}_{j}\right)}{(n+2)\alpha - 1} = \frac{\frac{N_j(n)}{n}-\frac{1}{\alpha n}}{\frac{n+2}{n}-\frac{1}{\alpha n}} \overset{a.s.}{\longrightarrow} X_3P_{j-2} \quad \text{ as } n \rightarrow \infty, $$
where $X_3 \sim \textnormal{Beta}(1-2\beta,3\beta)$ and $(P_i, i \geq 1) \sim \textnormal{GEM}(1-\beta,1-2\beta)$. Thus, we have obtained the almost sure limiting weight partition for the subtrees $(\tau_j, j \geq 0)$. The proof of (i) follows noting that $\sigma^{(n)} = \emptyset$ for all $n \geq 1$ almost surely when $\alpha=2$.
\end{proof}

We will establish the independence of $(X_0,X_1,X_2,X_3)$ and $(P_i,i\ge 1)$ in Proposition \ref{Marchal independence SCALING FACTOR theorem} to
fully specify the distribution of $(\xi_j,j\ge 0)$.

\subsection{Independent copies of Marchal's algorithm at the first branch point}

The proof of the following result is inspired by {\cite[Lemma~8]{curien2013stable}} in considering transition times at which a leaf is added into a subtree. However, we extend the result from \cite{curien2013stable} by considering transitions jointly over multiple subtrees. We restrict our considerations to $\alpha \in (1,2)$, i.e. the infinitary case. The result can easily be extended to the Brownian case $\alpha=2$.

\begin{proposition} \label{Marchal independence Theorem}
For $n\geq 1$ and $i\geq 0$, we have $\tau_i^{(N^{-1}_i(n))} \overset{d}{=} \mathbf{T}(n)$. That is, at transition times in which a leaf is added into the $i$-th subtree, it evolves as Marchal's algorithm with parameter $\alpha \in (1,2)$ with initial edge $e_i$. The sigma-field generated by
$$ \left(\tau_i^{(N^{-1}_i(n))}, n \geq 1\right)_{i \geq 0 }$$
is independent of the sigma-field generated by $\left(N_i(n), n\!\geq\!1, i\!\geq\!0\right)$. Consequently, $\left(\tau_i, i\!\geq\!0\right)$ are independent. Furthermore, $\left(\tau_i, i \geq 0\right)$ is independent of $\left(N_i(n), n \geq 1, i\geq 0\right)$. 
\end{proposition} 

\begin{proof} From Proposition \ref{Marchal mass distribution prop}, we have $K_n \rightarrow \infty$ a.s. as $n \rightarrow \infty$. In particular, for all $i \geq 0$ and $n \geq 1$, $N^{-1}_i(n) < \infty$ a.s.. We assume this holds henceforth. It suffices to show the independence of the sigma-fields generated by $$\left(N_i(n), n \geq 1,i\geq 0\right)\qquad \text{and} \qquad \left(\tau_i^{(N^{-1}_i(n))}, n \geq 1 \right)_{0 \leq i \leq m+2},$$ respectively, where $m \geq 0$ is arbitrary but fixed. 

Consider a given time $n \geq N^{-1}_{m+2}(1)$. Conditional on a leaf being added to the $i$-th subtree for $0 \leq i \leq m+2$, we have the dynamics of Marchal's algorithm with parameter $\alpha$ by the weight-leaf relation in Lemma \ref{Marchal total weight prop}. Likewise, the transition in the other components, not including the $i$-th subtree for $0 \leq i \leq m+2$, follows the correct conditional distributions of Marchal's algorithm. This proves the distributional identity $\tau_i^{(N^{-1}_i(n))} \overset{d}{=} \mathbf{T}(n)$ at transition times in the $i$-th subtree for $0 \leq i \leq m+2$. 

Let $M > 1$ be arbitrary, but fixed, and denote the natural filtration of $\left(N_i(n),i\geq 0\right)_{n \geq 1}$ by $\left(\mathcal{F}_n\right)_{n \geq 1}$. Note that for any fixed $n \geq 1$, $\left(N_i(n), i \geq 0 \right)$ is almost surely a vector with finitely many non-trivial entries. Define $T := \max_{i = 0,\ldots,m+2} N_i^{-1}(M)$, which is a stopping time with respect to $\left(\mathcal{F}_n\right)_{n \geq 1}$. By assumption, $T < \infty$ a.s.. Conditional on $\mathcal{F}_T$ (which is the same as conditioning on relative weights on subtrees until time $T$), we have factorisation of tree shape probabilities into tree shape probabilities for the respective subtrees cut at $V_2$. In particular, given $\mathcal{F}_T$, the tree shapes of $$\left(\tau_i^{(N^{-1}_i(n))}, 1 \leq n \leq M \right)_{0 \leq i \leq m+2}$$ are independent. Furthermore, on the event $\lbrace T = t \rbrace$, conditioning on the sigma-field generated at a later time $k \geq t$ does not affect the tree shapes under consideration. Hence, the hypotheses in Lemma \ref{Marchal independence technical Lemma} are fulfilled. Let $\mathbf{t}_i^{(n)}$ be some given leaf-labelled trees with $n$ leaves and a root. Then,
\allowdisplaybreaks[1]
\begin{align}
& \mathbb{P}\left(\tau_i^{(N^{-1}_i(n))} = \mathbf{t}_i^{(n)}, 1 \leq n \leq M , 0 \leq i \leq m+2 ~\bigg\rvert~ \mathcal{F}_{\infty} \right) \nonumber \\
& =\mathbb{P}\left(\tau_i^{(N^{-1}_i(n))} = \mathbf{t}_i^{(n)}, 1 \leq n \leq M , 0 \leq i \leq m+2 ~\bigg\rvert~ \mathcal{F}_{T} \right) \nonumber     \\
& =\mathbb{P}\left(\tau_i^{(N^{-1}_i(M))} = \mathbf{t}_i^{(M)}, 0 \leq i \leq m+2 ~\bigg\rvert~ \mathcal{F}_{T} \right) \label{Marchal independence equation 3} \\
& = \prod_{i=0}^{m+2} \mathbb{P}\left(\tau_i^{(N^{-1}_i(M))} = \mathbf{t}_i^{(M)}~\bigg\rvert~ \mathcal{F}_{T} \right)  \label{Marchal independence equation} \\
& = \prod_{i=0}^{m+2} \mathbb{P}\left(\mathbf{T}(M) = \mathbf{t}_i^{(M)}~\bigg\rvert~\mathcal{F}_{T} \right) \label{Marchal independence equation 2} \\
& = \prod_{i=0}^{m+2} \mathbb{P}\left(\mathbf{T}(M)= \mathbf{t}_i^{(M)}\right), \label{Marchal independence equation 4} 
\end{align} 
where \eqref{Marchal independence equation 3} holds since $\tau_i^{(N^{-1}_i(M))}$ determines $\tau_i^{(N^{-1}_i(n))}$ for all $1 \leq n \leq M$, \eqref{Marchal independence equation} holds by Proposition \ref{Marchal tree shape prob}, and \eqref{Marchal independence equation 4} follows since there is no dependence on $\mathcal{F}_{\infty}$ in evaluating \eqref{Marchal independence equation 2} and we are conditioning over an almost surely finite number of discrete random variables. Furthermore, since the final expression does not depend on $\mathcal{F}_{\infty}$, the sigma-field of $$\left(\tau_i^{(N^{-1}_i(n))}, 1 \leq n \leq M \right)_{0 \leq i \leq m+2}$$ is independent of $\mathcal{F}_\infty$. Letting $M \rightarrow \infty$, and recalling $m \geq 0$ is arbitrary, the claimed independence of the $\sigma$-fields follows.

As the collection $(\tau_i, i \geq 0)$ is measurably constructed from $ (\tau_i^{(N^{-1}_i(n))}, n \geq 1)_{i \geq 0},$ it is independent of $(N_i(n), n \geq 1, i\geq 0)$. Dropping the conditioning in \eqref{Marchal independence equation 2}, we get that $\big(\tau_i^{(N^{-1}_i(n))}, 1 \leq n \leq M\big)$, $0 \leq i \leq m+2$, are independent. Thus, in the limit as $M \rightarrow \infty$, $\big(\tau_i^{(N^{-1}_i(n))}, n \geq 1\big)$, $0 \leq i \leq m+2$, are independent. As $\tau_i$ is measurably constructed from $(\tau_i^{(N^{-1}_i(n))}, n \geq 1)$ for each $0 \leq i \leq m+2$, $\big(\tau_i, 0 \leq i \leq m+2\big)$ are independent. Let $m \rightarrow \infty$ to conclude that $(\tau_i, i \geq 0)$ are independent.  
\end{proof}

\begin{proposition} \label{Marchal independence SCALING FACTOR theorem}
For $\alpha \in (1,2)$, the random variables $\left(X_0,X_1,X_2,X_3\right)$ and $\left(P_i, i \geq 1\right)$ as defined in Proposition \ref{Marchal mass distribution prop} are independent. In particular, this fully specifies their joint distribution. 
\end{proposition}
\begin{proof} Recall that $N_{\sigma}(n)$ denotes the number of leaves (excluding $V_2$) in $\sigma^{(n)}$, and define $N_{\sigma}^{-1}(n) := \inf \lbrace k \geq 0\colon N_{\sigma}(k) = n \rbrace$. We claim that the sigma-field generated by $\left(N_j \left(N_{\sigma}^{-1}(n)\right), n \geq 1\right)_{j \geq 3}$ is independent of the sigma-field generated by $$\left(N_0^{-1}(n),N_1^{-1}(n),N_2^{-1}(n),N_{\sigma}^{-1}(n), n \geq 1\right).$$ Let $(\mathcal{F}_n)_{n \geq 1}$ denote the natural filtration of $(N_0^{-1}(n),N_1^{-1}(n),N_2^{-1}(n),N_{\sigma}^{-1}(n))_{n \geq 1}$. It suffices to prove that the sigma-field generated by $(N_j \left(N_{\sigma}^{-1}(n)), n \geq 1\right)_{3 \leq j \leq m}$ is independent of $\mathcal{F}_{\infty}$, where $m \geq 3$ is arbitrary but fixed. By Proposition \ref{Marchal mass distribution prop}, almost surely, $N^{-1}_{\sigma}(n) < \infty$ for all $n \geq 1$. For $M \geq 1$ arbitrary but fixed, $T := N^{-1}_{\sigma}(M)$ is an almost surely finite stopping time with respect to $(\mathcal{F}_n)_{n \geq 1}$. Consider the random variables $(N_j (N_{\sigma}^{-1}(n)), 1 \leq n \leq M)_{3 \leq j \leq m}$. On the event $\lbrace T = t \rbrace$, conditioning on the sigma-field generated at a later time $k \geq t$ does not affect conditional expectations. Hence, by Lemma \ref{Marchal independence technical Lemma}, for all non-negative integers $l_j(n)$, 
\begin{align*}
\mathbb{P} &\left(N_j \left(N_{\sigma}^{-1}(n)\right)= l_j(n), 1 \leq n \leq M , 3 \leq j \leq m ~\bigg\rvert~ \mathcal{F}_{\infty} \right) \\
& =\mathbb{P}\left(N_j \left(N_{\sigma}^{-1}(n)\right)= l_j(n), 1 \leq n \leq M , 3 \leq j \leq m ~\bigg\rvert~ \mathcal{F}_{T} \right) \nonumber     \\
& = \mathbb{P}\left(N_j \left(N_{\sigma}^{-1}(n)\right)= l_j(n), 1 \leq n \leq M , 3 \leq j \leq m \right). \nonumber
\end{align*} 
The last equality follows, since conditional on a leaf being added to $\sigma^{(n)}$ at time $n+1$, the process of adding leaves to each subtree within $\sigma^{(n)}$ is modelled by a CRP with parameters $(1-\beta,1-2\beta)$, see Proposition \ref{Marchal mass distribution prop}. Furthermore, it does not depend on the times at which the leaf is added. Since we are conditioning over an almost surely finite number of discrete random variables, we may drop conditioning on $\mathcal{F}_T$. This implies that the sigma-field generated by $$\left(N_j \left(N_{\sigma}^{-1}(n)\right), 1 \leq n \leq M \right)_{3 \leq j \leq m}$$ is independent of $\mathcal{F}_{\infty}$ for all $M \geq 1$. Let $M \rightarrow \infty$ to conclude that the sigma-field generated by $\left(N_j \left(N_{\sigma}^{-1}(n)\right), n \geq 1\right)_{3 \leq j \leq m}$ is independent of $\mathcal{F}_{\infty}$, as desired. Since we may rewrite equation \eqref{Marchal mass distrn prop EQN 2} to get
$$ \frac{n}{N^{-1}_{\sigma}(n)} \overset{a.s.}{\longrightarrow} X_3 \quad \text{ as } n \rightarrow \infty,$$
we conclude that $X_3$ is $\mathcal{F}_{\infty}$-measurable. Likewise, $X_0, X_1$ and $X_2$ are $\mathcal{F}_{\infty}$-measurable. From \eqref{Marchal mass distrn prop EQN 1}, 
$$ \frac{N_{i+2}\left(N_{\sigma}^{-1}(n)\right)}{n} \overset{a.s.}{\longrightarrow} P_i \quad \text{ as } n \rightarrow \infty,$$
for all $i \geq 1$. So, $(P_i, i \geq 1)$ is measurable with respect to the sigma-field generated by $\left(N_j \left(N_{\sigma}^{-1}(n)\right), n \geq 1\right)_{j \geq 3}$. The desired result follows. 
\end{proof}
We finally obtain the main result regarding the self-similarity of Marchal's algorithm, which proves the self-similarity property of $\alpha$-stable trees when decomposed at the first branch point. 

\begin{theorem} \label{Marchal Main Theorem}
For any $\alpha \in (1,2]$, the limiting trees $(\tau_i, i \geq 0)$ in Marchal's algorithm are independent. Furthermore, they are independent of their scaling factors. For each subtree $\tau_i^{(n)}$, let $d_i^{(n)}$ denote the graph distance and $\mu_i^{(n)}$ the empirical mass measure on its leaves.  
\begin{enumerate}
\item[\rm{(i)}] If $\alpha = 2$, for each $i \in \lbrace 0,1,2 \rbrace$, we have the convergence 
$$ \left(\tau_i^{(n)},\frac{d_i^{(n)}}{2 n^{1/2}},\mu_i^{(n)}\right) \overset{a.s.}{\longrightarrow} \left( \tau_i,\xi_i^{1/2}d^{(\infty)}_i, \xi_i\mu^{(\infty)}_i \right) \quad \text{ as } n \rightarrow \infty,$$ in the Gromov--Hausdorff--Prokhorov topology, where $( \tau_i, d^{(\infty)}_i, \mu^{(\infty)}_i ), i \geq 1$, are i.i.d.\ with  $$\left( \tau_i, d^{(\infty)}_i, \mu^{(\infty)}_i\right)\,{\buildrel d \over =}\, \left(\mathcal{T}_{2},d_{2},\mu_{2}\right), \quad i=0,1,2,$$ and $\left(\xi_0,\xi_1,\xi_2\right) \sim \textnormal{Dir}\left(1/2,1/2,1/2\right)$ is
independent of $(\tau_0,\tau_1,\tau_2)$. 
\item[\rm{(ii)}] If $\alpha \in (1,2)$, for each $i \geq 0$, we have the convergence 
$$ \left(\tau_i^{(n)},\frac{d_i^{(n)}}{\alpha n^{\beta}},\mu_i^{(n)}\right) \overset{a.s.}{\longrightarrow} \left( \tau_i, \xi_i^\beta d^{(\infty)}_{i}, \xi_i\mu^{(\infty)}_{i} \right) \quad \text{ as } n \rightarrow \infty,$$ in the Gromov--Hausdorff--Prokhorov topology, 
where $( \tau_i, d^{(\infty)}_i, \mu^{(\infty)}_i ), i \geq 1,$ are i.i.d.\ with  $$\left( \tau_i, d^{(\infty)}_i, \mu^{(\infty)}_i\right)\,{\buildrel d \over =}\, \left(\mathcal{T}_{\alpha},d_{\alpha},\mu_{\alpha}\right), \quad i \geq 0,$$ 
and $\xi_i = X_i$ for $i \in \lbrace 0,1,2 \rbrace$ and $\xi_{j+2} = X_3 P_j$ for $ j \geq 1$, with $\left(X_0,X_1,X_2,X_3\right) \sim \textnormal{Dir}\left(\beta,\beta,\beta,1-2\beta\right)$ and $(P_i,i \geq 1)^{\downarrow} \sim \textnormal{PD}(1-\beta,1-2\beta)$ independent. 
\end{enumerate}
\end{theorem} 
\begin{proof} The almost sure convergence in the rescaled subtrees arises by applying Proposition \ref{Marchal GHP Curien and Haas Theorem} and Proposition \ref{Marchal mass distribution prop}. The independence between the limiting subtrees comes immediately from Theorem \ref{Marchal independence Theorem}. The arguments in Lemma \ref{Marchal total weight prop} and Proposition \ref{Marchal mass distribution prop} show that the limiting proportion of weights is measurably constructed from $(N_i(n), n \geq 1, i \geq 0)$. Hence, by Theorem \ref{Marchal independence Theorem}, the limiting subtrees are independent of their scaling factors. The distribution of $(\xi_i, i \geq 0)$ for $\alpha \in (1,2)$ is fully specified by Proposition \ref{Marchal independence SCALING FACTOR theorem}.  
\end{proof} 

The results of Theorem \ref{Marchal Main Theorem} agree with similar decompositions of the BCRT at a branch point in Aldous \cite[Theorem~2]{aldous1994recursive}, Albenque and Goldschmidt \cite[Section~1.4]{ag15}, and Croydon and Hambly \cite[Lemma~6]{ch08}, where the branch point is uniquely determined by a uniformly chosen point according to the mass measure within each of the three subtrees. We point out that Albenque and Goldschmidt deal with an unrooted BCRT, while Croydon and Hambly's construction uses a doubly-marked rooted BCRT. Our construction thus far does not require a notion of a mass measure (even though we have chosen to include the mass measure in our statements), but rather a single marked point in each subtree.

\subsection{Formal specification of the concatenation operation}

After verifying that the subtrees $(\tau_i, i \geq 0)$ are rescaled versions of $\mathcal{T}_{\alpha}$ in the limit with the required independences, the next step is to show that the concatenation operation induced by Marchal's algorithm is well-defined and measurable as an operation on $\mathbb{T}_{\rm m}$. To this end, we adapt the general setup and terminology from \cite{rw16}.
Let\vspace{-0.2cm} $$\Xi := \bigg\lbrace (x_0,x_1,x_2,x_3 p_j, j \!\geq\! 1)\colon x_0, x_1, x_2, x_3 \!\geq\! 0,\,\sum_{i=0}^{3} x_i \!=\! 1,\, p_1 \!\geq\! p_2 \!\geq\! \cdots \!\geq\! 0,\,\sum_{j=1}^{\infty} p_j \!=\! 1 \bigg\rbrace.\vspace{-0.2cm}$$ For notational convenience, we write
$ 
    \xi_i=\left\{ 
\begin{array}{ll}
    x_i& \text{if } i \in \lbrace 0,1,2 \rbrace,\\
    x_3 p_{i-2} & \text{otherwise}.
\end{array}\right.\vspace{0.1cm}
$
 
Set $\Xi^{*} := \Xi \times \mathbb{T}_{\rm m}^{\infty}$ as in \eqref{RDE prodspace} and recall that $\mathbb{T}_{\rm m}$ is the set of ${\rm GH}^{\rm m}$-equivalence classes of marked compact rooted $\mathbb{R}$-trees.
Note that in our case, $N = \inf\lbrace i \geq 0\colon \xi_{i+1}=0 \rbrace$ with the convention $\inf \emptyset = \infty$, so that $N$ is a function of $\xi = \left(\xi_i, i \geq 0\right) \in \Xi$. Furthermore, in the case of $\alpha$-stable trees, recall that $N = 2$ or $N = \infty$ almost surely. Hence, we drop dependence on $N$ in our notation. For $\beta \in \left(0,\frac{1}{2}\right]$, equip $\Xi^{*}$ with the metric\vspace{-0.1cm}
\begin{equation} \label{Xi* metric}
d_{\beta}\left(\left(\xi,\tau_i,i\!\geq\!0\right),\left(\xi',\tau'_i, i\!\geq\!0\right)\right):=\sup_{i\geq0} \left(\lvert \xi_i^{\beta} \!-\! {\xi'_i}^{\beta} \rvert \vee d_{\rm GH}^{\rm m}\left(\tau_i,{\tau'_i}\right) \vee d_{\rm GH}^{\rm m}\left(\xi_i^{\beta}\tau_i,{\xi'_i}^{\beta}{\tau'_i}\right) \right),\vspace{-0.1cm}
\end{equation}
where $\xi = \left(\xi_i, i\geq 0\right) \in \Xi$, $\xi' = \left(\xi_i', i \geq 0 \right)\in \Xi$, and $(\tau_i,d_i,\rho_i,x_i), ({\tau'_i},d',{\rho'_i},{x'_i})$ are representatives of ${\rm GH}^{\rm m}$-equivalence classes in $\mathbb{T}_{\rm m}$, with shorthand $\xi_i^\beta\tau_i$ meaning that all distances of $\tau_i$ are reduced by the factor $\xi_i^\beta$. However, as $d_{\rm GH}^{\rm m}$ only depends on ${\rm GH}^{\rm m}$-equivalence classes, our metric $d_{\beta}$ also only depends on ${\rm GH}^{\rm m}$-equivalence classes. Hence, we may define $d_{\beta}$ on $\Xi^{*}$ and denote by $\tau$ any representative of the ${\rm GH}^{\rm m}$-equivalence class of $(\tau,d,\rho,x)$. 

\begin{proposition} \label{Xi* Polish prop}
$\left(\Xi^{*},d_{\beta}\right)$ is a Polish metric space.
\end{proposition}  
\begin{proof} This can be proved following the lines of the proof of \cite[Proposition~3.1]{rw16}.
\end{proof}

We now formally define our concatenation operator. Let $\xi \in \Xi$ and let $\left(\tau_i,d_i,\rho_i,x_i\right)$ be representatives of ${\rm GH}^{\rm m}$-equivalence classes in $\mathbb{T}_{\rm m}$ for $i \geq 0$. Define the \textit{concatenated tree} $\left(\tau',d',\rho',x'\right)$ as follows. 

\begin{enumerate}
\item Let $\widetilde{\tau}' := \coprod_{i\geq0} \tau_i$ be the disjoint union of trees. Let $\sim_{c}$ be the equivalence relation on $\widetilde{\tau}'$ in which $\rho_i \sim_{c} x_0$ for all $i \geq 1$. Define $\tau' := \widetilde{\tau}'/\sim_{c}$. Write $\psi_{c}$ for the canonical projection from $\widetilde{\tau}'$ onto $\tau'$. 
\item Define $d'$ as the metric induced on $\tau'$ under $\psi_c$ by the metric $\widetilde{d}'$ on $\widetilde{\tau}'$ such that 
\begin{equation} \label{concatenation metric}
    \widetilde{d}'(u,v)= 
\begin{cases}
    \xi^{\beta}_i d_i(u,v)& \text{if } u,v \in \tau_i, i \geq 0,\\
    \xi^{\beta}_0 d_0(u,x_0)+\xi^{\beta}_j d_j(\rho_j,v)& \text{if } u \in \tau_0 \text{ and } v \in \tau_j, j \neq 0, \\
    \xi^{\beta}_i d_i(u,\rho_i) + \xi^{\beta}_0 d_0(x_0,v) & \text{if } u \in \tau_i \text{ and } v \in \tau_0, i \neq 0,\\
    \xi^{\beta}_i d_i(u,\rho_i)+\xi^{\beta}_j d_j(\rho_j,v) & \text{if } u \in \tau_i \text{ and } v \in \tau_j, i,j \neq 0.\\
\end{cases}
\end{equation}
%Let $d'$ be the metric induced on $\tau'$ under $\psi_{c}$ by $\widetilde{d}'$.
\item Retain $x' = \psi_{c}\left(x_1\right)$ as our marked point in $\tau'$ and set $\rho' = \psi_{c}\left(\rho_0\right)$ as the root of $\tau'$. 
\end{enumerate}
We illustrate this construction in Figure \ref{fig:dissertConcatenation}. \pagebreak

\begin{figure}[t]
  $\;$\hfill\includegraphics[width=12cm]{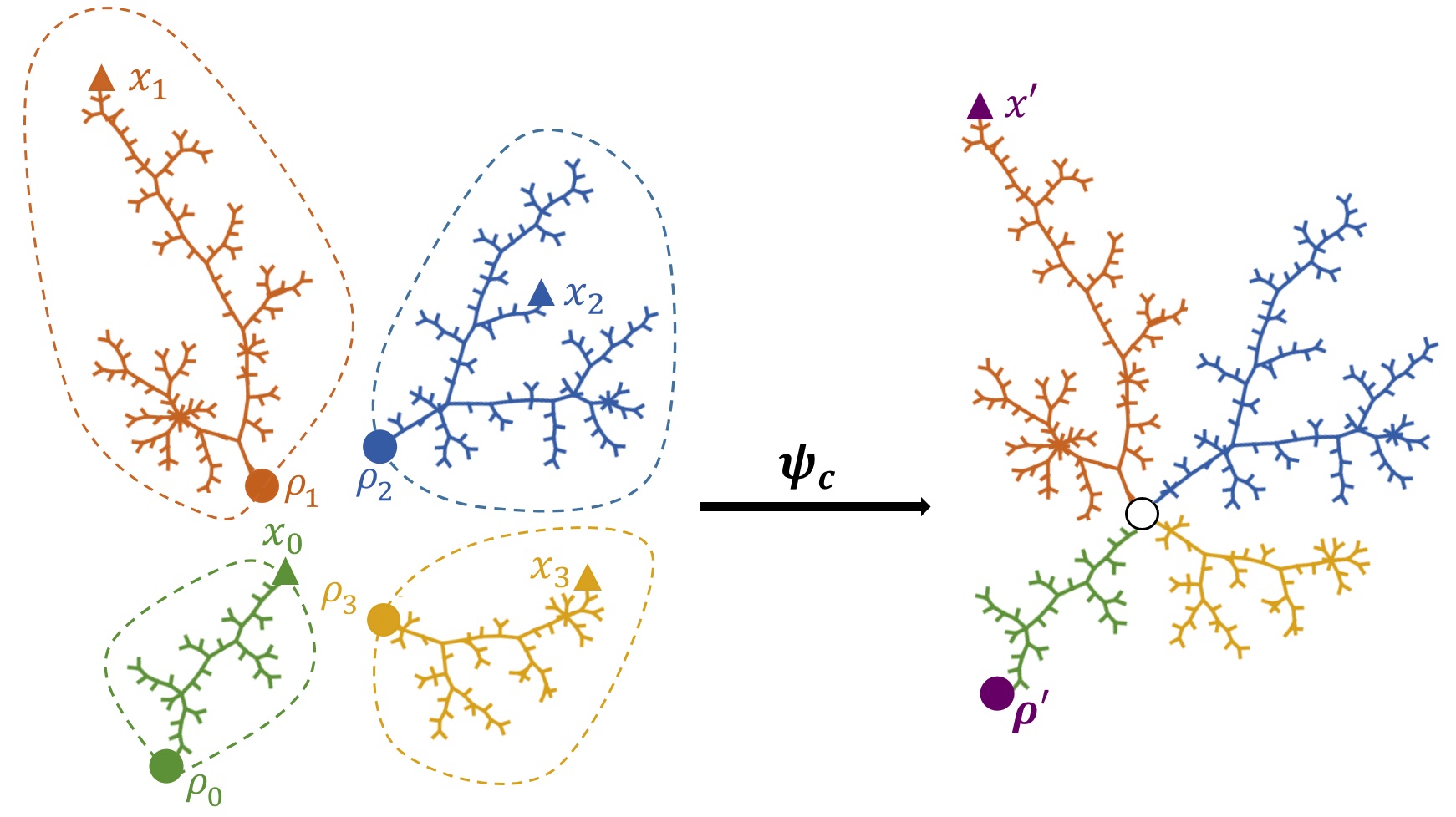}\hfill$\;$
  \caption{Construction of concatenated tree from 4 marked trees, rescaling is not shown; extracted from simulations (original simulations courtesy of Igor Kortchemski).}
  \label{fig:dissertConcatenation}
\end{figure}

By virtue of this construction, the ${\rm GH}^{\rm m}$-equivalence class of $\left( \tau',d',\rho',x' \right)$ only depends on the ${\rm GH}^{\rm m}$-equivalence classes of $\left(\tau_i,d_i,\rho_i,x_i\right)$ for $i \geq 0$. Thus, it makes sense to define $C_{\beta} \subseteq \Xi^{*}$ as the set of elements $\kappa = \left(\xi,\tau_i, i \geq 0\right) \in \Xi^{*}$ such that the concatenated tree $\left(\tau',d',\rho',x'\right)$ formed by any equivalence class representatives of $\left(\left(\tau_i,d_i,\rho_i,x_i\right), i \geq 0\right)$ is compact. Equip $\mathbb{T}_{\rm m}$ and $\Xi^{*}$ with their respective Borel sigma-algebras, $\mathcal{B}(\mathbb{T}_{\rm m})$ and $\mathcal{B}(\Xi^{*})$. The \textit{concatenation operator} $g_{\beta}\colon \Xi^{*} \rightarrow \mathbb{T}_{\rm m}$ is, 
\begin{equation} \label{concatenation operator}
    g_{\beta}(\kappa)= 
\begin{cases}
    \left(\tau',d',\rho',x'\right)& \text{if } \kappa\in C_{\beta},\\
    \left(\left\lbrace x' \right\rbrace,0,x',x'\right) & \text{otherwise},
\end{cases}
\end{equation} 
where $(\lbrace x' \rbrace,0,x',x')$ denotes the equivalence class of a trivial one-point rooted tree.  

\begin{proposition} \label{measurability prop}
The map $g_{\beta}\colon \Xi^{*} \rightarrow \mathbb{T}_{\rm m}$ is $\mathcal{B}(\Xi^{*})$-measurable. 
\end{proposition} 
\begin{proof} The proof can be adapted from \cite[Proposition~3.2]{rw16}.\end{proof}

\subsection{First main result: the RDE satisfied by the stable tree}

We now deduce our main theorem in this section. 
\begin{theorem} \label{RDE for alpha-stable tree THM} 
The marked $\alpha$-stable tree $(\mathcal T_\alpha, d_\alpha, \mu_\alpha, x_\alpha)$ with $x_\alpha \sim \mu_\alpha$ satisfies the RDE
\begin{equation} \label{RDE for alpha-stable tree}
\mathcal{T}_{\alpha} \overset{d}{=} g_{\beta}\left(\xi,\mathcal{T}_i,i \geq 0\right) \end{equation} 
on $\mathbb{T}_{\rm m}$, where $(\mathcal{T}_i, i \geq 0)$ is a sequence of independent copies of $\mathcal T_\alpha$, independent of $\xi = \left(X_0,X_1,X_2,X_3P_j, j \geq 1 \right) \in \Xi$, and where the following holds.
\begin{itemize}
\item If $\alpha = 2$, then $\xi_{j+2} = X_3P_j = 0$ almost surely for all $j \geq 1$ and $\left(X_0,X_1,X_2\right) \sim \textnormal{Dir}\left(1/2,1/2,1/2\right)$.
\item If $\alpha \in (1,2)$, then $\left(X_0,X_1,X_2,X_3\right) \sim \textnormal{Dir}\left(\beta,\beta,\beta,1-2\beta\right)$ and $(P_j, j \geq 1) \sim \textnormal{PD}(1-\beta,1-2\beta)$, where $\left(X_0,X_1,X_2,X_3\right)$ and $(P_j, j \geq 1)$ are independent.
\end{itemize}
In other words, the law of the marked $\alpha$-stable tree $\varsigma_\alpha^{\rm m}$ satisfies the fixpoint equation 
$$ \eta= \Phi_{\beta}\left(\eta\right) \label{RDEdis} $$
on $\mathcal{P}(\mathbb{T}_{\rm m})$, where $\Phi_{\beta}\colon \mathcal{P}\left(\mathbb{T}_{\rm m}\right) \rightarrow \mathcal{P}\left(\mathbb{T}_{\rm m}\right)$ is the mapping on $\mathcal{P}(\mathbb{T}_{\rm m})$ induced by \eqref{RDE for alpha-stable tree}, and where we recall that $\mathcal{P}(\mathbb{T}_{\rm m})$ denotes the set of Borel probability measures on $\mathbb{T}_{\rm m}$.
\end{theorem}
\begin{proof} Recall that for the subtrees involved in the recursive application of Marchal's algorithm, we regarded $V_2$ as a root and marked the first leaf in the $i$-th subtree for each $i \geq 1$. We regarded $A_0$ as the root for the overall tree, and $V_2$ as a marked leaf for the $0$-th subtree. Thus, our construction using Marchal's algorithm agrees with the concatenation operator $g_{\beta}$ acting on the subtrees. Theorem \ref{Marchal Main Theorem} gives the required independences and the distribution of $\xi = \left(\xi_i, i \geq 0\right)$. Proposition \ref{measurability prop} ascertains the measurability of $g_{\beta}$. 
\end{proof}

In general, the marked $\alpha$-stable tree is not the only fixpoint of \eqref{RDE for alpha-stable tree}. Observe that if the metrics $d_i$ of (the representatives of) $\left(\tau_i,d_i,\rho_i,x_i\right) \in \mathbb{T}_{\rm m}$ in \eqref{concatenation metric} were multiplied by some constant $c > 0$, then the concatenated tree will also have its metric $d'$ multiplied by $c$. Furthermore, if the original concatenated tree were a marked compact rooted $\mathbb{R}$-tree, then so would the concatenated tree with metric multiplied by $c$. Thus, since $\left(\mathcal{T}_\alpha,d_\alpha,\rho_\alpha,x_\alpha\right)$ is a distributional fixpoint of \eqref{RDE for alpha-stable tree}, so is $\left(\mathcal{T}_\alpha,c d_\alpha,\rho_\alpha,x_\alpha\right)$ for any $c > 0$. 

\begin{remark}\rm \label{RDE alpha counterexample Prop} 
There also exist solutions to RDE \eqref{RDE for alpha-stable tree} with infinite $1/\beta$-th height moment. This can be shown by 
grafting mass-less length-$y$ branches onto a stable tree with intensity proportional to $y^{-1-1/\beta}dy\mu(dx)$, see e.g. \cite{broutin2016self} 
and \cite{ag15} for such constructions in the context of related RDEs with finite concatenation operations -- the arguments there are not affected 
by the change of setting here. We will establish uniqueness of the solution to  \eqref{RDE for alpha-stable tree} up to multiplication of distances 
by a constant, under suitable constraints on height moments. 
\end{remark}

\section{Uniqueness and attraction for a general RDE on $\mathbb{T}_{\rm m}$}\label{secuniq}

\subsection{An RDE on $\mathbb{T}$ and associated constructions in $\mathbb{T}_{\rm w}$ of \cite{rw16}}\label{secRW}

In \cite{rw16}, we established a recursive construction method for CRTs by successively replacing the atoms of a random \textit{string of beads}, that is, a random interval $[0,L]$ for some $L>0$ equipped with a random discrete probability measure $\mu$, with scaled independent copies of itself. More general versions of the CRT construction using so-called $\textit{generalised strings}$ were established to capture multifurcating  self-similar CRTs. We briefly recap our construction, and refer to \cite{rw16} for more details.

Strings of beads can be represented in the form $([0,l], (x_i)_{i \in I}, (q_i)_{i \in I})$ where $l > 0$ denotes the length of the interval, and $x_i \in [0,l]$, $i \in I$, are distinct and describe the locations of the atoms with respective masses $q_i \geq 0, i \in I$, $\sum_{i \in I}q_i =1$, where $I$ is some countable index set. The concept of a string of beads can be generalised by allowing for non-distinct $x_i$'s% and a mass measure $\lambda$ on $[0,l]$ such that $\lambda([0,l])=1-\sum_{i \in I} q_i$
. We call $\left([0,l], (x_i)_{i \in I}, (q_i)_{i \in I}\right)$ a \textit{generalised string}. 
%In this paper, we will always have $\lambda=0$. In this case, we also refer to $\left([0,l], (x_i)_{i \in I}, (q_i)_{i \in I}\right):=\left([0,l], (x_i)_{i \in I}, (q_i)_{i \in I},0\right)$ as a generalised string.
The following theorem is a (slightly simplified) version of the main result in \cite{rw16}.

\begin{theorem} \label{rec con} \label{moments of T check} Let $\beta \in (0,\infty)$ and $p>1/\beta$. Consider a random generalised string 
  $\zeta=(\check{\mathcal T}_0, (\check{X}^{(0)}_i)_{i \in I },(\check{Q}^{(0)}_i)_{i \in I })$ 
  with length $L >0$ such that $\mathbb{E}[L^p] < \infty$, and atom masses $0\leq \check{Q}^{(0)}_i < 1$ a.s.\ for all $i \in I$ and such that 
  $\sum_{i \in I}\check{Q}^{(0)}_i=1$ a.s.. For $n\geq 0$, to obtain 
  $$\left(\check{\mathcal T}_{n+1}, \left(\check{X}_i^{(n+1)}\right)_{i \in I }, \left(\check{Q}^{(n+1)}_i\right)_{i \in I }\right)$$ 
  conditionally given $(\check{\mathcal T}_{n}, (\check{X}_i^{(n)})_{i \in I}, (\check{Q}^{(n)}_i)_{i \in I})$, attach to each 
  $\check{X}_i^{(n)} \in \check{\mathcal T}_n$ an independent isometric copy of $\zeta$ with metric rescaled by $(\check{Q}^{(n)}_i)^\beta$ and 
  atom masses rescaled by $\check{Q}^{(n)}_i$. 

  Let $\check{\mu}_n=\sum_{i \in I} \check{Q}^{(n)}_i \delta_{\check{X_i}^{(n)}},  n \geq 0.$ Then there exists a random weighted $\mathbb R$-tree 
  $(\check{\mathcal T}, \check{\mu})$ such that 
  $$\lim_{n \rightarrow \infty} \left(\check{\mathcal T}_n, \check{\mu}_n \right)=\left(\check{\mathcal T}, \check{\mu}\right) \qquad \text{a.s.}$$
  in the Gromov--Hausdorff--Prokhorov topology in $\mathbb T_{\rm w}$. Furthermore, 
  $\mathbb{E}[{\rm ht}(\check{\mathcal T})^p]< \infty$ for all $p < p^*:=\sup\{p \geq 1\colon \mathbb{E}[L^p] < \infty\}$. 
\end{theorem}

The convergence in Theorem \ref{rec con} holds in particular in the Gromov--Hausdorff sense when we omit mass measures. In fact, this construction
is naturally carried out in the Banach space $\ell_1(\bU)$, $\bU:=\bigcup_{n\ge 0}\bN^n$, which is a variant of Aldous's $\ell_1(\bN)$ since 
$\bU$ is countable. So embedded, the convergence holds with respect to the Hausdorff metric (or a Hausdorff--Prokhorov metric) for compact subsets 
(equipped with a probability measure) of $\ell_1(\bU)$, as a consequence of the arguments of \cite{rw16}. %It will be convenient to denote 
In particular, the $\alpha$-stable tree was characterised as the limit in the case of a \textit{$\beta$-generalised string} for 
$\beta=1-1/\alpha \in (0,1/2]$, that is, a generalised string of the form 
$$\left(\left[0,L\right], \left(X_i\right)_{i \geq 1}, \left(P_i\right)_{i \geq 1} \right)$$ 
where, for $(Q_m, m \geq 1) \sim {\rm PD}(\beta, \beta)$ independent of i.i.d.\ $(R_j^{(m)}, j\geq 1) \sim {\rm PD}(1-\beta, -\beta)$, $m \geq 1$, 
the atom sizes are given via 
$$ \left(P_i, i \geq 1\right)=\left(Q_m R_j^{(m)}, j \geq 1, m \geq 1\right)^\downarrow,$$ 
and the atom locations are defined via i.i.d.\ Unif$([0,1])$-variables $(U_m, m \geq 1)$ and 
$$L:=\lim_{m \rightarrow \infty} m \Gamma(1-\beta)Q_m^\beta, \quad X_i=LU_m \text{ if } P_i=Q_mR_j^{(m)}, \qquad i \geq 1.$$

\subsection{Second main result: uniqueness and attraction for the new RDE}

We now turn to the uniqueness and attraction of the fixpoints in \eqref{RDE for alpha-stable tree}. By Theorem 
\ref{RDE for alpha-stable tree THM} and Remark \ref{RDE alpha counterexample Prop}, uniqueness will only hold up to multiplication by a constant 
and under additional moment conditions on tree heights. As our setup works for more general $\xi \in \Xi$, we will broaden our scope, and consider 
the RDE \eqref{RDE for alpha-stable tree} in a less specific setting.

It will be useful to work in the framework of a recursive tree process, as defined in Section \ref{RDE subsection}. Let us consider a sequence of 
i.i.d.\ $\mathbb R$-trees with one marked leaf with distribution $\eta$ on $\mathbb{T}_{\rm m}$, and an i.i.d.\ family of sequences of scaling factors 
$(\xi_{\mathbf{u}i}, i \geq 0), \mathbf{u} \in \mathbb{U},$ with some distribution $\nu$ on $\Xi$, where we recall the Ulam--Harris notation 
$\mathbb{U}=\bigcup_{n\geq 0}\mathbb{N}^n$. 

For $n\geq 1$, we would like to study the distribution $\Phi_\beta^n(\eta)$ of $\cT_n:=\tau^{(n)}_\emptyset$, where
\begin{equation} \label{treesequ} 
  \tau^{(n)}_{\mathbf{u}}:=g_\beta\left( (\xi_{\mathbf{u}i},i\ge 0), \left(\tau^{(n)}_{\mathbf{u} i}, i \geq 0\right)\right),
  \quad\mathbf{u} \in \mathbb{N}^k, \quad k=n,\ldots,1,
\end{equation}
for $\tau^{(n)}_{\mathbf{u} i} \sim \eta, i \geq 0$, $\mathbf{u} \in \mathbb{N}^n$, i.i.d.. Note that this setup induces a recursive tree process, and, in particular, a recursive tree framework $(((\xi_{\mathbf{u}i}, i\geq 0), \mathbf{u} \in \mathbb{U}),g_\beta)$.

Furthermore, let ${\mathcal P}_\infty(\mathbb{T}_{\rm m})\subset \mathcal P(\mathbb{T}_{\rm m})$ be defined as
$$ {\mathcal P}_\infty \left(\mathbb{T}_{\rm m}\right):=\left\{\eta \in \mathcal P\left(\mathbb{T}_{\rm m}\right)\colon \mathbb{E}\left[{\rm ht}\left(\mathcal T\right)^p\right]<\infty \text{ for all } p > 0 \text{ where } \left(\mathcal T,d,\rho,x\right) \sim \eta  \right\}.$$ 

Our main result in this section is as follows.

\begin{theorem}\label{uni and attr} For any $\Xi$-valued random variable $\xi\!=\!(\xi_i,i\!\ge\! 0)$ such that 
  $\mathbb{P}(\xi_0\!+\!\xi_1\! <\! 1)\!=\!1$ and $\mathbb P(\xi_0>0, \xi_1 >0)=1$, choose $\beta \in (0,1)$ such that 
  $\mathbb{E}[\xi_0^\beta + \xi_1^\beta]=1$. Then, for any $\eta \in \mathcal P_\infty(\mathbb{T}_{\rm m})$ with $h:=\mathbb{E}[d(\rho, x)]$ for 
  $(\mathcal T, d, \rho, x) \sim \eta$, 
  $$\Phi^n_\beta\left(\eta\right) \rightarrow \eta^*_h \text{ weakly as } n \rightarrow \infty,$$
  where $\eta^*_h$ is the unique fixpoint of $\Phi_\beta$ in $\mathcal{P}_\infty(\mathbb{T}_{\rm m})$ with $\mathbb{E}[d^*(\rho^*\!,x^*)]\!=\!h$ for $(\mathcal T^*\!, d^*\!, \rho^*\!, x^*)\! \sim\! \eta^*_h$.
\end{theorem}

Note that the function $f\colon [0,1] \rightarrow (0, \infty)$, $\beta \mapsto \mathbb{E}[\xi_0^\beta + \xi_1^\beta]$ is continuous with $f(0)=2$ 
and $f(1) < 1$ when $\mathbb{P}(\xi_0+\xi_1 < 1)=1$. Hence, there is always some $\beta \in (0,1)$ such that $f(\beta)=1$ in the situation of 
Theorem \ref{uni and attr}.

The uniqueness and attractiveness of the marked $\alpha$-stable tree in \eqref{RDE for alpha-stable tree} is a direct consequence of 
Theorem \ref{uni and attr}.

\begin{corollary} Let $\alpha \in (1,2]$ and set $\beta:=1-1/\alpha \in (0,1/2]$. Furthermore, let $\xi=(X_0, X_1, X_2, X_3P_j, j \geq 1)$ for 
  independent $(X_0,X_1, X_2, X_3)\sim \rm{Dir}(\beta, \beta, \beta, 1-2\beta)$ and $(P_j, j \geq 1) \sim \rm{PD}(1-\beta, 1-2\beta)$. Then the law 
  $\varsigma_\alpha^{\rm m}$ of the marked $\alpha$-stable tree is the unique fixpoint of $\Phi_\beta$ on $\mathcal{P}_\infty(\mathbb{T}_{\rm m})$ 
  with $\mathbb{E}[d(\rho, x)]= \alpha \Gamma(\beta)/\Gamma(2\beta)$ for $(\mathcal T, d, \rho,  x) \sim \eta$, 
  $\eta \in \mathcal{P}_\infty(\mathbb{T}_{\rm m})$. Furthermore, for any $\eta \in \mathcal{P}_\infty(\mathbb{T}_{\rm m})$ with 
  $\mathbb{E}[d(\rho, x)]=h$ for $(\mathcal T, d, \rho, x) \sim \eta$, we have 
  $$\Phi_\beta^n\left(\eta\right) \rightarrow \varsigma^{\rm m}_{\alpha,h} \text{ weakly as } n \rightarrow \infty,$$ 
  where $\varsigma^{\rm m}_{\alpha,h}$ denotes the distribution of the marked $\alpha$-stable  tree with distances scaled by 
  $h/(\alpha \Gamma(\beta)/\Gamma(2\beta))$. 
\end{corollary}

\begin{proof} Apply Theorem \ref{uni and attr} with the specific distribution for $\xi$, and $\beta=1-1/\alpha$. Furthermore, recall from Theorem 
  \ref{RDE for alpha-stable tree THM} that the marked $\alpha$-stable tree is a fixpoint of the resulting RDE, is well-known to have height moments 
  of all orders (e.g.\ from its construction via Theorem \ref{rec con}), and from Section \ref{GPU and CRP 
  subsection} that the distance between the root and a uniformly sampled leaf of the $\alpha$-stable tree has distribution ${\rm ML}(\beta, \beta)$
  scaled by $\alpha$, which has mean $\alpha \Gamma(\beta)/\Gamma(2\beta)$ by \eqref{Mittag--Leffler distribution}.
\end{proof}

To prove Theorem \ref{uni and attr}, we first focus on the case when $\eta$ is supported on the space of probability measures on $\textit{trivial}$ 
trees, that is, single branch trees with a root and exactly one leaf (which is marked). We further require that the length of such a tree has 
moments of orders $p > 0$. Specifically, we consider
$$\mathbb{T}_{\rm m}^{\rm tr}:=\left\{(\mathcal T, d,0,y) \in \mathbb{T}_{\rm m}\colon \mathcal T=\llbracket0,y\rrbracket, y > 0 \right\}.%\pagebreak
$$
For most of the proof, we will work in the special case of $\mathbb{T}_{\rm m}^{\rm tr}$-valued initial distributions: 

\medskip

\noindent{\bf Assumption\,($\!$A$\!$): }$\eta\!\in\!\mathcal P_\infty\!\left(\mathbb{T}_{\rm m}^{\textrm{tr}} \right)\!:=\!
  \left\{\eta\!\in\!\mathcal P\!\left(\mathbb{T}_{\rm m}^{\textrm{tr}}\right)\!\colon\mathbb{E}[({\rm ht}(\mathcal T))^p]\!<\!\infty 
  \text{\,for\,all\,}p \!>\! 0 
  \text{\,where\,} \mathcal T\!\sim\! \eta \right\}$.

\medskip 

Under Assumption (A), we will show the convergence of the spine from the root to the marked point in the RDE (Section \ref{secspine}), the convergence of subtrees spanned by leaves up to recursion depth $k$ (Section \ref{secktreeconv}), the CRT limit as $k\rightarrow\infty$ (Section \ref{secktoinfty}) and establish that the RDE is attractive, pulling threads together via a tightness argument (Section \ref{secattr}). We
finally strengthen this to lift Assumption (A) and complete the proof of Theorem \ref{uni and attr}.

%From now on, let $\eta \in \mathcal P_\infty (\mathbb{T}_{\rm m}^{\textrm{tr}})$, and 
For the remainder of this section, we write $(\mathcal T_n, n \geq 0)$ for the sequence of trees constructed in \eqref{treesequ} from $\tau^{(n)}_{\mathbf{u}j} \sim \eta, \mathbf{u} \in \mathbb{N}^n, j \geq 0$. We write $Y_{\mathbf{u}j}:={\rm ht}(\tau^{(n)}_{\mathbf{u}j})$, $\mathbf{u} \in \mathbb{U}$, $j \geq 0$.

\subsection{The spine from the root to the marked point in the RDE}\label{secspine}

We first study an $\mathcal L^p$-bounded martingale arising from the fixpoint equation in Theorem \ref{uni and attr}, which tracks the length of
the spine from the root to the marked point.

\begin{lemma} \label{xi martingale} Let $\xi$ be a $\Xi$-valued random variable with $\bP(\xi_0>0,\xi_1>0)=1$. Let $\beta\in(0,1]$ such that 
  $\bE[\xi_0^\beta+\xi_1^\beta]=1$, let $(\xi_{\mathbf{u}j}, j \geq 0), \mathbf{u} \in \mathbb{U}$, be i.i.d.\ with the same distribution as $\xi$, and define
\begin{equation}   \overline{\xi}_\mathbf{u}:=\xi_{u_1}\xi_{u_1u_2}\cdots\xi_{u_1\ldots u_n}, \quad \mathbf{u}=u_1\ldots u_n \in \mathbb{N}^n, \quad n \geq 1. \end{equation}
   Then the process
  \begin{equation} L_n=\sum_{\mathbf{u} \in\{0,1\}^n}\overline{\xi}_\mathbf{u}^\beta\label{ximartingale} \end{equation}
 is a mean-1 martingale that converges a.s. and in 
%mean $\mathbb{E}[L_n]=1$ for all $n \geq 1$, and is bounded in 
$\mathcal L^p$ for all $p\ge 1$.  
%Furthermore, there is $L_\infty$ such that, as $n \rightarrow \infty$,
%  $$L_n \rightarrow L_\infty \text{ a.s. } \quad \text{ and } \quad L_n \rightarrow L_\infty  \text{ in } \mathcal L^p \text{ for all } p \geq 1.$$ 
\end{lemma}

\begin{proof} It is straightforward to show that $(L_n, n \geq 0)$ is a martingale with $\bE[L_n]=1$ for all $n \geq 1$. So we focus on the $\mathcal L^p$-boundedness. For $p=1$, we have for all $n\ge 1$,
  $$\bE\left[L_n\right]=\sum_{\mathbf{u}\in\{0,1\}^n}\bE\left[\xi_{u_1}^\beta\right]\cdots\bE\left[\xi_{u_1\ldots u_n}^\beta\right]=\left(\bE\left[\xi_0^\beta\right]+\bE\left[\xi_1^\beta\right]\right)^n=1.$$
  Inductively, if for all $j\le p-1$ and $n\ge 1$, we have $\bE[L_n^j]\le f(j)$, then for all $n\ge 1$,
  \begin{align*}
    \bE\left[L_n^p\right]
	  &=\sum_{\mathbf{u}^{(1)},\ldots,\mathbf{u}^{(p)}\in\{0,1\}^n}\bE\left[\overline{\xi}_{\mathbf{u}^{(1)}}^\beta\cdots\overline{\xi}_{\mathbf{u}^{(p)}}^\beta\right]\\
	  &=\sum_{\mathbf{v}\in\{0,1\}^n}\bE\left[\overline{\xi}_\mathbf{v}^{p\beta}\right]
	  +\sum_{k=0}^{n-1}\sum_{\mathbf{v}\in\{0,1\}^k}\bE\left[\overline{\xi}_\mathbf{v}^{p\beta}\right]\sum_{j=1}^{p-1}\binom{p}{j}\bE\left[\xi_0^{j\beta}\xi_1^{(p-j)\beta}\right]\\	  
  & \ \quad\qquad\qquad\qquad  \times  	  \bE\!\left[\sum_{\substack{\mathbf{w}^{(1)},\ldots,
	   \mathbf{w}^{(j)} \\ \in\{0,1\}^{n-k-1}}} \!\!\!\left(\overline{\xi}_{\mathbf{w}^{(1)}}^\beta\cdots\overline{\xi}_{\mathbf{w}^{(j)}}^\beta\right)\!\right]  
	   !\!\bE\!\left[\sum_{\substack{\mathbf{w}^{(j+1)},\ldots,\mathbf{w}^{(p)} \\ \in\{0,1\}^{n-k-1}}} \!\!\!\left(\overline{\xi}_{\mathbf{w}^{(j+1)}}^\beta\cdots\overline{\xi}_{\mathbf{w}^{(p)}}^\beta\right)\!\right].
  \end{align*}
  Specifically, we split the sum over $\mathbf{u}^{(1)},\ldots,\mathbf{u}^{(p)}$ according to the number $k$ of initial entries that are common to all 
  $\mathbf{u}^{(1)},\ldots,\mathbf{u}^{(p)}$ and according to the number $j$ of entries in the $(k+1)$-st place of $\mathbf{u}^{(1)},\ldots,\mathbf{u}^{(p)}$ that equal 0. For each
  $k$ and $j$, there are $\binom{p}{j}$ ways to choose which $j$ they are. By symmetry, the contribution is the same as if they are $1,\ldots,j$,
  so that we write the sum as a sum over $$\mathbf{u}^{(1)}=\mathbf{v}0\mathbf{w}^{(1)},\ldots,\mathbf{u}^{(j)}=\mathbf{v}0\mathbf{w}^{(j)},\mathbf{u}^{(j+1)}=\mathbf{v}1\mathbf{w}^{(j+1)},\ldots,\mathbf{u}^{(p)}=\mathbf{v}1\mathbf{w}^{(p)}.$$ 
  By the induction hypothesis, we can further bound $\bE[L_n^p]$ above by
 \begin{align*}
      &\sum_{k=0}^n\sum_{\mathbf{v}\in\{0,1\}^k}\bE\left[\overline{\xi}_\mathbf{v}^{p\beta}\right]\sum_{j=1}^{p-1}\binom{p}{j}f(j)f(p-j)\\
%	  &\le\left(\sum_{j=1}^{p-1}\binom{p}{j}f(j)f(p-j)\right)\sum_{k=0}^n\left(\bE\left[\xi_0^{p\beta}+\xi_1^{p\beta}\right]\right)^k\\
	  &\le \left(\sum_{j=1}^{p-1}\binom{p}{j}f(j)f(p-j)\right)\frac{1}{1-\bE\left[\xi_0^{p\beta}+\xi_1^{p\beta}\right]}=:f(p)<\infty.
  \end{align*}
This completes the proof by the 
%The second part of the theorem follows from the $\mathcal L^p$-boundedness of $(L_n, n \geq 1)$ and the 
Martingale Convergence Theorem.
\end{proof}

\subsection{Convergence of subtrees spanned by leaves up to depth $k$}\label{secktreeconv}

For the following, it will be useful to represent the trees $\cT_n:=\tau_{\emptyset}^{(n)}$ of (\ref{treesequ}) in such a way that we can talk 
about ``the subtree of $\cT_n$ spanned by the leaves up to depth $k$''. Let us introduce notation for these leaves under the Assumption (A): denote 
by $\Sigma_{n,\mathbf{u}}$ the endpoint of the trivial tree $\tau_\mathbf{u}^{(n)}$ when repeatedly rescaled and finally used to build $\cT_n$. 
Then the leaves of $\cT_n$ up to depth $k$, together with the branch points up to depth $k$, are given by the set of $\Sigma_{n,\mathbf{u}}$ for 
$\mathbf{u}=u_1\cdots u_n\in\bN^n$, with $u_{k+1}=\ldots=u_n=1$.

\begin{proposition} \label{Tk's} Suppose Assumption {\rm(A)} holds and $\cT_n:=\tau_{\emptyset}^{(n)}$ in the setting of (\ref{treesequ}), $n\ge 0$. 
  Let $k \in \mathbb{N}$. For $n \geq k$, let $\mathcal T_n^k$ be the subtree of 
  $\mathcal T_n$ spanned by the root and the leaves up to depth $k$. We consider $\Sigma_{n,11\cdots1}$ as the respective marked point. Then there 
  is an increasing sequence of marked trees $(\mathcal T^k, k \geq 0)$ such that, for all $k \geq 0$, 
  $$\mathcal T_n^k \rightarrow \mathcal T^k \text{ in probability as } n \rightarrow \infty$$ 
  in the marked Gromov--Hausdorff topology. 
\end{proposition}

\begin{proof} For $k=0$, $\mathcal T_n^0$ is a trivial one-branch tree with a root and a marked leaf, and total length 
$$\widetilde{L}_n^0=\sum_{\mathbf{u} \in \{0,1\}^n} \overline{\xi}_{\mathbf{u}}^\beta Y_{\mathbf{u}}.$$
Recall the martingale $(L_n, n \geq 1)$ from \eqref{ximartingale} and denote its limit by $L_\infty$. Let $m:=\mathbb{E}[Y_\emptyset]$, and note that,
\begin{align*}
\bE\left[\left(\widetilde{L}^0_n-mL_n\right)^2 \right] 
&= \bE \left[\left( \sum_{\mathbf{u} \in \{0,1\}^n}  \overline{\xi}_{\mathbf{u}}^\beta \left( Y_\mathbf{u} - m\right) \right)^2 \right] \\
&= \sum_{\mathbf{u} \in \{0,1\}^n} \sum_{\mathbf{v} \in \{0,1\}^n}
\bE \left[ \overline{\xi}_{\mathbf{u}}^\beta  \overline{\xi}_{\mathbf{v}}^\beta \right]\bE \left[ \left( Y_\mathbf{u} - m\right)  \left( Y_\mathbf{v} - m\right)  \right]\\
&=\sum_{\mathbf{u} \in \{0,1\}^n} 
\bE \left[  \overline{\xi}_{\mathbf{u}}^{2\beta} \right]\bE \left[ \left( Y_\mathbf{u} - m\right)^2 \right]
= {\rm Var}\left(Y\right) \left( \bE\left[ \xi_{0}^{2\beta} + \xi_{1}^{2\beta} \right]\right)^n \rightarrow 0
\end{align*}
as $n \rightarrow \infty$, where we used the facts that $ Y_\mathbf{u}$ and $ Y_\mathbf{v}$ are independent for $\mathbf{u} \neq \mathbf{v}$, and $\bE[\xi_0^{2\beta} + \xi_1^{2\beta}] < 1$ as $0< \xi_0, \xi_1 < 1$ a.s.. 
Therefore, $\widetilde{L}_n^0 \rightarrow \widetilde{L}^0_\infty:= m\cdot L_\infty$ in $\mathcal L^2$ and almost surely as $n \rightarrow \infty$.

Under Assumption {\rm (A)}, the $Y_\mathbf{u}$ also have finite $p$-th moment for all $p\ge 3$ and splitting $p$-fold sums as in the proof of 
Lemma \ref{xi martingale}, it is straightforward to strengthen this convergence to $\mathcal L^p$-convergence.

Now, let $k\geq 1$, and note that the shapes of $\mathcal T_n^k$ and $\mathcal T_k$ coincide for all $n \geq k$. Let $\widetilde{L}^k_{n,\mathbf{u}}, \mathbf{u} \in \mathbb{N}^k$, denote the lengths of the edges of $\mathcal T_n^k$ using obvious notation, i.e. 
$$\widetilde{L}^k_{n,\mathbf{u} }:= \sum_{\mathbf{v}  \in \{0,1\}^{n-k}} \overline{\xi}_{\mathbf{u} \mathbf{v} }^\beta Y_{\mathbf{u} \mathbf{v} }, \quad \mathbf{u}  \in \mathbb{N}^k.$$
Furthermore, let $\mathcal T^k$ have the same shape and the same marked leaf as $\mathcal T_k$ with edge lengths $\widetilde{L}^k_{\infty,\mathbf{u} }, \mathbf{u}  \in \mathbb{N}^k$, given by 
$$\widetilde{L}^k_{\infty,\mathbf{u} }= \lim_{t \rightarrow \infty} \sum_{\mathbf{v}  \in \{0,1\}^t} \overline{\xi}_{\mathbf{u} \mathbf{v} }^\beta Y_{\mathbf{u} \mathbf{v} }, \quad \mathbf{u} \in \mathbb{N}^k,$$
which exists a.s.\ as a $\overline{\xi}_\mathbf{u} ^\beta$-scaled copy of $\widetilde{L}^0_\infty$, independent for $\mathbf{u}\in\mathbb{N}^k$.

Hence, for each $k \geq 0$, the differences 
$\left\lvert{\widetilde{L}_{n,\mathbf{u} }^k-\widetilde{L}^k_{\infty,\mathbf{u} }}\right\rvert$, $\mathbf{u} \in \mathbb{N}^k$, 
are $\overline{\xi}_\mathbf{u} ^\beta$-scaled independent copies of $\lvert \widetilde{L}_{n}^0-\widetilde{L}^0_\infty \rvert$. Therefore, for $p \geq 1/\beta$, as every leaf of $\cT_n^k$ or $\cT^k$ is at most $2^k$ edges from the root and from another leaf, by \eqref{marked distortion},
\begin{align*} 
\mathbb{E}\left[ \left(d^{\rm m}_{\rm GH}\left(\mathcal T_n^k, \mathcal T^k\right)\right)^p\right]  \leq 2^{pk}\mathbb{E}\left[ \max_{\mathbf{u}  \in \mathbb{N}^k}  \left\lvert \widetilde{L}_{n,\mathbf{u} }^k-\widetilde{L}^k_{\infty,\mathbf{u} } \right\rvert ^p    \right]
&\leq 2^{pk} \sum_{\mathbf{u}  \in \mathbb{N}^k} \mathbb{E}\left[  \left( \left\lvert \widetilde{L}_{n,\mathbf{u} }^k-\widetilde{L}^k_{\infty,\mathbf{u} } \right\rvert \right)^p    \right]\\ 
&= 2^{pk} \sum_{\mathbf{u}  \in \mathbb{N}^k}\mathbb{E}\left[ \overline{\xi}_\mathbf{u} ^{p\beta}\right] \mathbb{E}\left[ \left\lvert \widetilde{L}_{n}^0-\widetilde{L}^0_\infty\right\rvert^p\right].
\end{align*}
Since $\sum_{\mathbf{u}  \in \mathbb{N}^k} \mathbb{E}\left[ \overline{\xi}_\mathbf{u} ^{p\beta}\right] < \infty$ for $p \geq 1/\beta$ and 
$\widetilde{L}_{n}^0 \rightarrow \widetilde{L}^0_\infty$ in $\mathcal L^p$ as $n\rightarrow \infty$, we conclude that, for any $\epsilon > 0$,
$$\lim_{n \rightarrow \infty} \mathbb{P} \left(d^{\rm m}_{\rm GH}\left(\mathcal T_n^k, \mathcal T^k\right) > \epsilon\right) 
\leq \lim_{n \rightarrow \infty}  \epsilon^{-p}\mathbb{E}\left[ \left(d^{\rm m}_{\rm GH}\left(\mathcal T_n^k, \mathcal T^k\right)\right)^p\right] 
=0.$$
Hence, $\mathcal T_n^k \rightarrow \mathcal T^k$ in probability in the marked Gromov--Hausdorff topology as $n \rightarrow \infty$.
\end{proof}

\subsection{The CRT limit of $\mathcal{T}^k$ as $k\rightarrow\infty$}\label{secktoinfty}

Next, we want to prove the convergence of $\mathcal T^k$ as $k\rightarrow \infty$. To this end, we need to identify a suitable candidate for the limit. We employ the recursive construction method for CRTs as described in Section \ref{secRW}. Define a generalised string
\begin{equation}	\zeta=\left( \left[0, \widetilde{L}^0_\infty \right], \left(X_\mathbf{u}\right)_{\mathbf{u} \in \mathbb{U}^*}, \left(Q_\mathbf{u}\right)_{\mathbf{u} \in \mathbb{U}^*}\right) \label{string for T} \end{equation}
where $\mathbb{U}^*:=\bigcup_{n \geq 0} \{0,1\}^n \times \{2,3,\ldots\}$, $\widetilde{L}^0_\infty$ is given above, and $(X_\mathbf{u})_{\mathbf{u} \in \mathbb{U}^*}, (Q_\mathbf{u})_{\mathbf{u} \in \mathbb{U}^*}$ are defined by dyadically splitting $\widetilde{L}^0$ as follows. See Figure
\ref{figstring} for an illustration.

\begin{itemize}

\item Let $Q_{i}:=\xi_i$, $i \geq 2$, and, for $\mathbf{u}=u_1\ldots u_n \in \mathbb{U}^*$, define
$$Q_{\mathbf{u}}:=\xi_{u_1} \xi_{u_1 u_2} \cdots \xi_{u_1 u_2 \ldots u_n}=\overline{\xi}_{u_1 u_2 \ldots u_n}.$$
Note that $0 \leq Q_{\mathbf u} < 1$ a.s. for all $\mathbf{u} \in \mathbb{U}^*$, $\sum_{\mathbf{u} \in \mathbb{U}^*} Q_{\mathbf{u}}=1$ a.s., and $\mathbb{E}\left[ \sum_{\mathbf{u} \in \mathbb{U}^*} Q_{\mathbf{u}}^{p\beta } \right] <1$ for all $p > 1/\beta$.

\item Define the locations $\left( X_{\mathbf{u}} \right)_{\mathbf{u} \in \mathbb{U}^*}$ of the atoms with respective masses $\left(Q_{\mathbf{u}} \right)_{\mathbf{u} \in \mathbb{U}^*}$ by $$X_{i}=\lim_{m \rightarrow \infty}  \sum_{(0 u_2 \ldots u_m) \in \{0\} \times \{0,1\}^{m-1}} \overline{\xi}_{0 u_2\ldots u_m}^\beta Y_{0u_2\ldots u_m},\quad i \geq 2, $$ and, for general $\mathbf{u}=(u_1\ldots u_n) \in \mathbb{U}^*$,
$$X_{u_1 \ldots u_n} \!=\! \lim_{m \rightarrow \infty}\! \left\{\! \sum_{\substack{(v_1 \ldots v_n) \in \{0,1\}^{n}\colon (v_1\ldots v_n) \prec (u_1 \ldots u_n) \\ (v_1, \ldots, v_{n-1}, v_n) \neq (u_1, \ldots, u_{n-1},1)}}\ \sum_{v_{n+1}\ldots v_{n+m} \in \{0,1\}^m}\! 
\overline{\xi}_{v_1\ldots v_{n+m}}^\beta Y_{v_1 \ldots v_{n+m}} \!\right\}$$
where $\prec$ denotes the lexicographic order, that is, $$(v_1\ldots v_n) \prec (u_1 \ldots u_n) \iff \exists t \geq 1 \text{ such that } \forall k < t\colon v_k = u_k \text{ and } v_t < u_t.$$ 
\end{itemize}
Noting in particular that this specifies $X_{\mathbf{u}i}=X_{\mathbf{u}2}$ for all $i\ge 2$ and each $\mathbf{u}2\in\mathbb{U}^*$, the scaled
lengths and dyadic splits to depth $k=3$ are illustrated in Figure \ref{figstring}.
\begin{figure}
  \begin{picture}(390,60)
    \put(30,15){\line(1,0){330}}
    \put(20,14){$\scriptstyle\rho$}
    \put(365,14){$\scriptstyle\Sigma_\emptyset$}

    \put(153,8){\line(0,1){14}}
    \put(150,0){$\scriptstyle X_i$}
    \put(93,10){\line(0,1){10}}
    \put(88,0){$\scriptstyle X_{0i}$}
    \put(283,10){\line(0,1){10}}
    \put(278,0){$\scriptstyle X_{1i}$}
    \put(53,12){\line(0,1){6}}
    \put(46,0){$\scriptstyle X_{00i}$}
    \put(128,12){\line(0,1){6}}
    \put(121,0){$\scriptstyle X_{01i}$}
    \put(213,12){\line(0,1){6}}
    \put(206,0){$\scriptstyle X_{10i}$}
    \put(313,12){\line(0,1){6}}
    \put(306,0){$\scriptstyle X_{11i}$}
    
    \put(40,13){\line(0,1){4}}
    \put(70,13){\line(0,1){4}}
    \put(110,13){\line(0,1){4}}
    \put(143,13){\line(0,1){4}}
    \put(190,13){\line(0,1){4}}
    \put(242,13){\line(0,1){4}}
    \put(295,13){\line(0,1){4}}
    \put(350,13){\line(0,1){4}}
    
    \put(34,14){\line(0,1){2}}
    \put(48,14){\line(0,1){2}}
    \put(60,14){\line(0,1){2}}
    \put(84,14){\line(0,1){2}}
    \put(103,14){\line(0,1){2}}
    \put(117,14){\line(0,1){2}}
    \put(134,14){\line(0,1){2}}
    \put(147,14){\line(0,1){2}}
    \put(165,14){\line(0,1){2}}
    \put(200,14){\line(0,1){2}}
    \put(222,14){\line(0,1){2}}
    \put(255,14){\line(0,1){2}}
    \put(291,14){\line(0,1){2}}
    \put(302,14){\line(0,1){2}}
    \put(335,14){\line(0,1){2}}
    \put(357,14){\line(0,1){2}}
    
    \put(169,25){$\scriptstyle\widetilde{L}_{\infty,100}^3$}
    \put(166,27){\vector(-1,0){13}}
    \put(196,27){\vector(1,0){17}}
    \put(234,25){$\scriptstyle\widetilde{L}_{\infty,101}^3$}
    \put(231,27){\vector(-1,0){18}}
    \put(261,27){\vector(1,0){22}}
    \put(285,25){$\scriptstyle\widetilde{L}_{\infty,110}^3$}
    \put(322,25){$\scriptstyle\widetilde{L}_{\infty,111}^3$}
    \put(319,27){\vector(-1,0){6}}
    \put(349,27){\vector(1,0){11}}

    \put(42,35){$\scriptstyle\widetilde{L}_{\infty,00}^2$}
    \put(39,37){\vector(-1,0){9}}
    \put(65,37){\vector(1,0){28}}
    \put(117,35){$\scriptstyle\widetilde{L}_{\infty,01}^2$}
    \put(114,37){\vector(-1,0){21}}
    \put(140,37){\vector(1,0){13}}
    \put(202,35){$\scriptstyle\widetilde{L}_{\infty,10}^2$}
    \put(199,37){\vector(-1,0){46}}
    \put(225,37){\vector(1,0){58}}
    \put(302,35){$\scriptstyle\widetilde{L}_{\infty,11}^2$}
    \put(299,37){\vector(-1,0){16}}
    \put(325,37){\vector(1,0){35}}
    
    \put(86,45){$\scriptstyle\widetilde{L}_{\infty,0}^1$}
    \put(83,47){\vector(-1,0){53}}
    \put(106,47){\vector(1,0){47}}
    \put(276,45){$\scriptstyle\widetilde{L}_{\infty,1}^1$}
    \put(273,47){\vector(-1,0){120}}
    \put(296,47){\vector(1,0){64}}
    \put(149,55){$\scriptstyle\widetilde{L}_\infty^0$}
    \put(146,57){\vector(-1,0){116}}
    \put(163,57){\vector(1,0){197}}
  \end{picture}
  \caption{Dyadic structure of limiting branch lengths 
  $\widetilde{L}_{\infty,\mathbf{u}}^k=\widetilde{L}_{\infty,\mathbf{u}0}^{k+1}+\widetilde{L}_{\infty,\mathbf{u}1}^{k+1}$, for which 
  $\overline{\xi}_\mathbf{u}^{-\beta}\widetilde{L}_{\infty,\mathbf{u}}^k$, $\mathbf{u}\in\{0,1\}^k$ are i.i.d.\ for each $k\ge 1$, ordered in 
  lexicographical order; atom positions $X_{\mathbf{u}i}$ are between fragments $\widetilde{L}_{\infty,\mathbf{u}0}^n$ and $\widetilde{L}_{\infty,\mathbf{u}1}^n$, $\mathbf{u}\in\{0,1\}^n$, $n\ge 0$.}
  \label{figstring}
\end{figure}
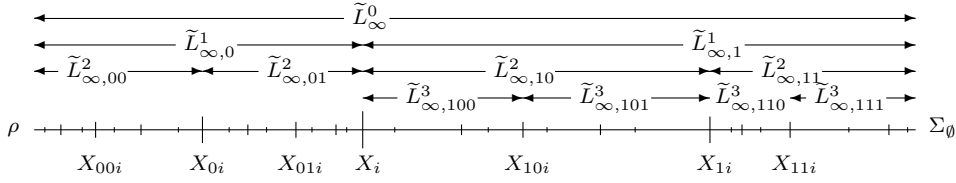

We now apply the recursive construction as outlined in Theorem \ref{rec con} to the generalised string $\zeta$, which results in an $\mathbb R$-tree $\mathcal T$, whose distribution we denote by $\eta^*$. 

\begin{proposition} \label{Construction of T} Let $\beta \in (0,1]$, and $p > 1/\beta$. Consider the generalised string $\zeta$ given by \eqref{string for T}. Apply the recursive construction described in Theorem \ref{rec con} to construct a sequence of random $\mathbb R$-trees $(\mathcal T_n^*, n \geq 0)$. Then $\mathcal T_n^*\rightarrow \mathcal T$ a.s. in the Gromov--Hausdorff topology for some random compact $\mathbb R$-tree $\mathcal T$ with $$\bE\left[{\rm ht}\left(\mathcal T\right)^p\right]< \infty \text{ for all } p > 0.$$
\end{proposition}

\begin{proof} This is a direct application of Theorem \ref{rec con}.
\end{proof}

It will be convenient to refer to the root as $\rho$ and the endpoints of the generalised strings as $\Sigma_{\mathbf{u}}$, $\mathbf{u}\in\bU$. 
Specifically, we denote by $\Sigma_\emptyset$ the endpoint of $\cT_0^*$. In the first step of the construction performed in 
Proposition \ref{Construction of T}, we attach branches to all branch points on the initial spine at once. The string has all the branch points 
$X_{\mathbf u}$, $\mathbf{u}\in\mathbb{U}^*=\bigcup_{j\ge 0}\{0,1\}^j\times\{2,3,\ldots\}$ placed as in the $n\rightarrow\infty$ limit. In 
particular, we have $\Sigma_{\mathbf{u}}\in\cT_1^*$ for all $\mathbf{u}\in\mathbb{U}^*$ and for more general $\mathbf{u}\in\mathbb{U}$, we have 
$\Sigma_{\mathbf{u}}\in\cT_n^*$ if and only if $\mathbf{u}=u_1\cdots u_j$ has at most $n$ entries $u_i\in\{2,3,\ldots\}$, i.e.\ at least $j-n$ 
entries $u_i\in\{0,1\}$.  
%As $(\mathcal T_n^*, n \geq 0)$ evolves, there is no change of the initial spine. The attached branches are scaled independent copies of the initial string.

In contrast, as $(\mathcal T_n, n \geq 0)$ evolves, the branch points on the initial spine are created successively, and distances change in each 
step as branches are replaced by two scaled branches in each step. We will further couple the vectors $(\xi_{\mathbf{u}i},i\ge 0)$, 
$\mathbf{u}\in\mathbb{U}$, of the construction of $(\mathcal T_n,n\ge 0)$, and the generalised strings $\zeta_{\mathbf{v}}$, 
$\mathbf{v}\in\mathbb{U}$. Specifically, we take $\widetilde{L}^k_{\infty,\mathbf{v}}$ as the length of $\zeta_{\mathbf{v}}$ and build
$((X_{\mathbf{v},\mathbf{u}})_{\mathbf{u}\in\mathbb{U}^*},(Q_{\mathbf{v},\mathbf{u}})_{\mathbf{u}\in\mathbb{U}^*})$ from the appropriate subfamilies 
of $((\xi_{\mathbf{u}i},i\ge 0),\mathbf{u}\in\mathbb{U})$. We will not require precise notation for these subfamilies, but we will exploit the
coupling and the independence of these subfamilies for all $\mathbf{v}\in\mathbb{N}^n$, $n\ge 0$, which is a consequence of the branching property 
of the recursive tree framework $((\xi_{\mathbf{u}i},i\ge 0),\mathbf{u}\in\mathbb{U})$.
%The relationship between $(\mathcal T_n^*, n \geq 0)$ and $(\mathcal T_n, n \geq 0)$ is reflected in the proof of Proposition \ref{Tk's}.

%We will exploit the coupling of $(\mathcal T_n, n \geq 0)$ with the sequence of i.i.d.\ copies of the generalised string $\zeta$ used in the 
%construction in Proposition \ref{Construction of T}, to obtain the following result. 

Indeed, we can represent $\cT$ like $\cT_n$, $n\ge 0$, in $\ell_1(\mathbb{U})$ in such a way that the convergence of Proposition \ref{Tk's} holds
for the Hausdorff metric on compact subsets of $\ell_1(\mathbb{U})$. Then, we have further a.s. convergence of $\Sigma_{n,\mathbf{u}}$ to limits that
we denote by $\Sigma_{\mathbf{u}}$, for all $\mathbf{u}\in\mathbb{U}$. Then the trees $\cT^k$ are spanned by 
$\Sigma_{\mathbf{u}}$, $\mathbf{u}\in\bigcup_{0\le j\le k}\bN^j$, while $\cT_n^*$ is spanned by 
$\Sigma_{\mathbf{u}}$, $\mathbf{u}=\mathbf{u}^{(1)}v_1\mathbf{u}^{(2)}v_2\cdots\mathbf{u}^{(n)}v_n\mathbf{u}^{(n+1)}$, $\mathbf{u}^{(1)},\ldots,\mathbf{u}^{(n+1)}\in\bigcup_{m\ge 0}\{0,1\}^m$, $v_1,\ldots,v_n\in\bN$. 

\begin{lemma} \label{T^k converges to T} Let $(\mathcal T^k, k\geq 0)$ be the sequence of trees from Proposition \ref{Tk's}, and let $(\mathcal T_n^*, n \geq 0)$ be the sequence of trees from Proposition \ref{Construction of T} with $\mathcal T_n^* \rightarrow \mathcal T$ a.s. as $n \rightarrow \infty$. Then $$ \mathcal T^k \rightarrow \mathcal T \text{ a.s. as } k \rightarrow \infty $$ in the marked Gromov--Hausdorff topology. 
\end{lemma}

\begin{proof} Since the sequence of trees $(\mathcal T^k,k \geq 0)$ is increasing and embedded in $\mathcal T$ with the same marked point, it remains to show that the almost 
sure limit of $\mathcal T^k$ is the whole of $\mathcal T$. 

Let $(\mathcal T_{\mathbf{u}j}, j \geq 2)$, $\mathbf{u} \in \bigcup_{t=0}^{\infty} \mathbb{N}^{k}\times \{0,1\}^t$, denote the connected components of $\mathcal T \setminus \mathcal T^k$, $k \geq 0$, where we write $(\mathcal T_{{u_1 \ldots u_n}j}, j \geq 2)$ for the subtrees of $\mathcal T \setminus \mathcal T^k$ rooted at the edge of $\mathcal T_k$ of length $\widetilde{L}_{\infty, {u_1 \ldots u_k}}^k$, $n \geq k$, using notation from Proposition \ref{Tk's}. Exploiting the fact that each $\mathcal T_{\mathbf{u}j}$ is a $\overline{\xi}_{\mathbf{u}j}$ scaled independent copy of $\mathcal T$, we obtain for $k \geq 0$ and $p > 1/\beta$,

\begin{align*} \bE\left[ \left( d_{\rm GH}^{\rm m}\left(\mathcal T^k, \mathcal T \right)\right)^p \right] &\leq \bE\left[ \left( \max_{\mathbf{u} \in \bigcup_{t=0}^{\infty} \mathbb{N}^{k}\times \{0,1\}^t, j \geq 2 } {\rm ht}\left(\mathcal T_{\mathbf{u}j} \right)\right)^p \right]\\
& \leq  \bE\left[\left({\rm ht}\left(\mathcal T\right)\right)^p\right] \sum_{\mathbf{u} \in \bigcup_{t=0}^{\infty} \mathbb{N}^{k}\times \{0,1\}^t, j\geq 2 } \bE\left[ \overline{\xi}_{\mathbf{u}j}^{p\beta} \right] \\
&\leq \bE\left[\left({\rm ht}\left(\mathcal T\right)\right)^p\right] \sum_{\mathbf{u} \in \bigcup_{t=0}^{\infty} \mathbb{N}^{k}\times \{0,1\}^t } \bE\left[ \overline{\xi}_{\mathbf{u}}^{p\beta} \right] \\
&\leq  \bE\left[\left({\rm ht}\left(\mathcal T\right)\right)^p\right]  \left( \bE\! \left[ \sum_{j \geq 0}\xi_j^{p\beta}  \right]\right)^{\!k} \!\bE \left[ \sum_{t=0}^\infty \left(\xi_0^{p\beta}+\xi_1^{p\beta}\right)^t \right]\!\rightarrow 0 \ \text{ as\,} k\!\rightarrow\!\infty
\end{align*}
as $\bE\left[{\rm ht}\left(\mathcal T\right)^p\right] < \infty$ and $\bE \left[\sum_{j\geq0} \xi_j^{p\beta} \right] <1$. Hence, for any $\epsilon > 0$ and $p > 1/\beta$, 
$$ \mathbb{P} \left(  d_{\rm GH}^{\rm m}\left(\mathcal T^k, \mathcal T \right) > \epsilon \right) \leq \epsilon^{-p} \bE \left[\left(  d_{\rm GH}^{\rm m}\left(\mathcal T^k, \mathcal T \right)\right)^p \right] \rightarrow 0$$
as $k \rightarrow \infty$. Therefore, due to the embedding of $(\mathcal T^k, k\geq 0)$ into $\mathcal T$ , $\mathcal T^k \rightarrow \mathcal T$ a.s. as $k\rightarrow \infty$.
\end{proof}

\subsection{Attraction of the RDE and the proof of Theorem \ref{uni and attr}}\label{secattr}

Next, we show that the supremum of the height moments of $\mathcal T_n$ is finite, employing the recursive construction of CRTs for a generalised string defined in a similar manner as in the discussion before Proposition \ref{Construction of T}.

\begin{lemma} \label{sup tree} Under Assumption {\rm(A)}, the sequence of trees $(\mathcal T_n, n \geq 0)$ satisfies \begin{equation} \label{bounded height moments}
\bE\left[ \sup_{n \geq 0} {\rm  ht} \left( \mathcal T_n \right)^p \right] < \infty \text{ for all } p > 0. \end{equation} \end{lemma}

\begin{proof} The idea of the proof is to construct a CRT $\widehat{\mathcal{T}}$ whose height dominates ${\rm ht}(\mathcal{T}_n)$ for all $n\ge 0$. 
  Indeed, we apply the recursive construction of CRTs (cf. the construction of $\mathcal T$) to the generalised string $\widehat{\zeta}$ obtained
  by modifying the definition of $\zeta$ in \eqref{string for T} by replacing $\lim_{m \rightarrow \infty}$ by $\sup_{m \geq 0}$ in the definition 
  of interval length and atom locations. In particular, the length of the interval is given by 
  $\sup_{n \geq 0} \sum_{\mathbf{u} \in \{0,1\}^n} \overline{\xi}_{\mathbf u}^\beta Y_{\mathbf{u}}$.

  This ensures that each atom is placed at the furthest position away from the root which appears in the course of the construction of 
  $\mathcal T_n, n \geq 0$. Hence, all distances between branch points, leaves and the root are larger than in any of the trees 
  $\mathcal T_n, n\geq 0$.

  Applying Theorem \ref{rec con} to the generalised string $\widehat{\zeta}$, we obtain a CRT $\widehat{\mathcal T}$ which has finite height 
  moments of all orders. By the underlying coupling, ${\textrm{ht}}(\mathcal T_n) \leq {\textrm{ht}}(\widehat{\mathcal T})$ for all $n \geq 0$, 
  i.e., the claim follows.
\end{proof}

\begin{corollary} \label{limit tree2} Consider the sequences of trees $(\mathcal T_n, n \geq 0)$ and $(\mathcal T^k_n, n \geq k)$, $k\geq 0$, where we recall that, for $n \geq k$, $\mathcal T_n^k$ is the subtree of $\mathcal T_n$ spanned by the root and the leaves up to depth $k$. Then, for any $\epsilon >0$, \begin{equation} \lim_{k \rightarrow \infty} \limsup_{n \rightarrow \infty} \mathbb{P}\left(d_{\rm GH}^{\rm m}\left(\mathcal T_n^k, \mathcal T_n\right) > \epsilon \right)=0.  \end{equation}
\end{corollary}

\begin{proof} Let $\mathcal T_{n,\mathbf{u}}^k\setminus\{\rho_{n,\mathbf{u}}^k\}, \mathbf{u} \in \bigcup_{t=0}^{n-k-1} \mathbb{N}^{k}\times \{0,1\}^t \times\{2,3,\ldots\}$, denote the subtrees of $\mathcal T_n \setminus \mathcal T_n^k$, $n \geq k+1$:
$$\mathcal T_n \setminus \mathcal T_n^k = \bigcup_{\mathbf{u} \in \bigcup_{t=0}^{n-k-1} \mathbb{N}^{k}\times \{0,1\}^t \times \{2,3,\ldots\}}  \mathcal T_{n,\mathbf{u}}^k\setminus\{\rho_{n,\mathbf{u}}^k\}.$$
Then, for any $\epsilon > 0$ and $p > 1/\beta$, 
\begin{align*}
\mathbb{P}\left(d_{\rm GH}^{\rm m}\left(\mathcal T_n^k, \mathcal T_n\right) > \epsilon \right) &\leq \epsilon^{-p} \bE \left[ \max_{\mathbf{u} \in \bigcup_{t=0}^{n-k} \mathbb{N}^{k}\times \{0,1\}^t  \times \{2,3,\ldots\}}   {\rm ht} \left(\mathcal T_{n,\mathbf{u}}^k\right)^p \right] \\
& \leq \epsilon^{-p} \sum_{\mathbf{u} \in \bigcup_{t=0}^{n-k-1} \mathbb{N}^{k}\times \{0,1\}^t  \times\{2,3,\ldots\}} \mathbb{E} \left[ \overline{\xi}_{\mathbf{u}}^{p\beta}\right] \bE \left[{\rm ht}\left(\mathcal T_{n-\lvert \mathbf{u}\rvert}\right)^p \right].
\end{align*}
By Lemma \ref{sup tree}, it remains to show that 
\begin{equation} \lim_{k \rightarrow \infty} 
\limsup_{n \rightarrow \infty} 
\sum_{\substack{\mathbf{u} \in \bigcup_{t=0}^{n-k-1} \mathbb{N}^{k} \times \{0,1\}^t  \times \{2,3,\ldots\}}} \bE \left[\overline{\xi}_{\mathbf{u}}^{p\beta}\right] = 0, \label{limsup} 
\end{equation}
First, note that the left-hand side of \eqref{limsup} is bounded above by 
\begin{equation} \label{upbound1} \lim_{k \rightarrow \infty } \sup_{n \geq k+1} \sum_{\mathbf{u} \in  \mathbb{N}^{k}} \sum_{t=0}^{n-k-1} \sum_{\mathbf{v} \in\{0,1\}^t  \times \{2,3,\ldots\} } \mathbb{E} \left[ \overline{\xi}_{ \mathbf{u}\mathbf{v} }^{p\beta}\right], \end{equation}
where we also slightly rewrote the expression. 
By the fact that $(\xi_{\mathbf{u}j}, j \geq 0), \mathbf{u} \in \mathbb{U}$, are i.i.d., we have $$\mathbb{E} \left[ \overline{\xi}_{ \mathbf{u}\mathbf{v} }^{p\beta}\right] =
\mathbb{E} \left[ \overline{\xi}_{ \mathbf{u}}^{p\beta}\right]
\mathbb{E} \left[ \overline{\xi}_{ \mathbf{v} }^{p\beta}\right] \leq \mathbb{E} \left[ \overline{\xi}_{ \mathbf{u}}^{p\beta}\right]
\mathbb{E} \left[ \overline{\xi}_{ \mathbf{v} }\right],$$
where we used $\overline{\xi}_{\mathbf{v}}<1$ a.s. and $p\beta > 1$ in the last inequality. 

Furthermore, as $\sum_{j \geq 0} \xi_{\mathbf{v}j}=1$, 
\begin{align*} \sum_{t=0}^{n-k-1}\sum_{\mathbf{v} \in\{0,1\}^t  \times \{2,3,\ldots\} } \mathbb{E} \left[ \overline{\xi}_{\mathbf{v} }\right] 
          \leq \sum_{t=0}^{n-k-1}\left(\bE\left[\xi_0+\xi_1\right]\right)^t
          \leq \sum_{t=0}^{\infty}\left(\bE\left[\xi_0+\xi_1\right]\right)^t= \left(1-\bE\left[\xi_0+\xi_1\right]\right)^{-1}
\end{align*}
where we also used the i.i.d.\ property of the $(\xi_{\mathbf{v}j},j\ge 0)$, $\mathbf{v} \in \bigcup_{t=0}^{n-k-1} \{0,1\}^t$, and $\bE\left[\xi_0+\xi_1\right] < 1$.
Hence,  \eqref{upbound1} can be further bounded above by 
\begin{equation} \label{upbound2}
  \left(1-\bE\left[\xi_0+\xi_1\right]\right)^{-1} \lim_{k \rightarrow \infty } 
    \sum_{\mathbf{u} \in  \mathbb{N}^{k} } \mathbb{E} \left[ \overline{\xi}_{ \mathbf{u} }^{p\beta}\right] 
  =\left(1-\bE\left[\xi_0+\xi_1\right]\right)^{-1}\lim_{k\rightarrow\infty}\left(\bE\left[\sum_{i\ge 0}\xi_i^{p\beta}\right]\right)^{k}. \end{equation} 

As $p\beta >1$ and $0 \leq \xi_i < 1$ a.s. for all $i \geq 0$, $\bE \left[ \sum_{i \geq 0} \xi_i^{p \beta}\right]<1$, and we conclude that \eqref{upbound2} is $0$.
\end{proof}

We are now ready to prove our final result. 

\begin{corollary} \label{limit tree} Under Assumption {\rm (A)}, let $(\mathcal T_n, n \geq 0)$ be as above, and let $\mathcal T$ be the tree from Proposition \ref{Construction of T}. We have the convergence $$\mathcal T_n \rightarrow \mathcal T \text{ in probability as } n \rightarrow \infty$$
in the marked Gromov--Hausdorff topology. 
\end{corollary}

\begin{proof} Let $\epsilon >0$, and use the triangle inequality twice to get, for $n \in \mathbb{N}$ and $k\le n$,
$$\mathbb{P}(d_{\rm GH}^{\rm m}(\mathcal T_n, \mathcal T) > 3\epsilon ) \leq 
\mathbb{P}(d_{\rm GH}^{\rm m}(\mathcal T_n, \mathcal T_n^k) > \epsilon ) +
\mathbb{P}( d_{\rm GH}^{\rm m}(\mathcal T_n^k, \mathcal T^k) > \epsilon ) +
\mathbb{P}( d_{\rm GH}^{\rm m}(\mathcal T^k, \mathcal T) > \epsilon ).$$

All three terms converge to $0$ as $n\rightarrow \infty$, and then $k \rightarrow \infty$, cf. Proposition \ref{Tk's}, Lemma \ref{T^k converges to T} and Corollary \ref{limit tree2}.
\end{proof}

Theorem \ref{uni and attr} is now a direct consequence of Corollary \ref{limit tree}.

\begin{pfofthm42} Let $\eta\in \mathcal P_{\infty}(\mathbb T_{ m})$ be a general distribution of a marked 
  $\mathbb R$-tree. For $(\mathcal T_0, d_0, \rho_0, x_0) \sim \eta$, we define the induced distribution 
  $\eta^\circ \in \mathcal P_\infty(\mathbb T_{\rm m}^{\rm tr})$ as the distribution of $\llbracket\rho_0,x_0\rrbracket$. We construct coupled 
  $(\cT_n,n\ge 0)$ and $(\cT_n^\circ,n\ge 0)$ from the same recursive tree framework $((\xi_{\mathbf{u}i},i\ge 0),\mathbf{u}\in\mathbb{U})$ and
  from coupled systems of i.i.d.\ $\eta$- and $\eta^\circ$-distributed trees, according to  \eqref{treesequ}, with $\mathcal{T}_0\sim\eta$ and
  $\mathcal{T}_0^\circ=\llbracket\rho_0,x_0\rrbracket\sim\eta^\circ$.  
  Then $\cT_0\setminus\cT_0^\circ$ consists of subtrees of heights bounded by ${\rm ht}(\cT_0)$. By construction, 
  $\cT_n\setminus\cT_n^\circ$ consists of subtrees of heights bounded by the maximum of $\overline{\xi}_{\mathbf{u}}$-scaled independent 
  copies of ${\rm ht}(\cT_0)$. Hence,
  $$\bE\left(\left(d_{\rm GH}^{\rm m}(\cT_n,\cT_n^\circ)\right)^p\right)
%  &\le\bE\left(\sup_{\mathbf{u}\in\bN^n,v\in\mathbb{U}^*}\left({\rm ht}(\cS_{n,\mathbf{u}\mathbf{v}})\right)^p\right)\\
%  &\le\bE\left(\sum_{\mathbf{u}\in\bN^n}\overline{\xi}_{\mathbf{u}}^{p\beta}\right)
%     \bE\left(\max_{\mathbf{v}\in\mathbb{U}^*}\left({\rm ht}(\mathcal{S}_{\mathbf{v}})\right)^p\right)\\
  \le\bE\left(\left({\rm ht}(\cT_0)\right)^p\right)\left(\bE\left(\sum_{j\ge 0}\xi_j^{p\beta}\right)\right)^n\rightarrow 0,
  $$
%letting $\mathcal T_{\max} \sim \eta_{\max}$, 
%$$\bE\left[\left( d_{\rm GH}^{\rm m}\left(\mathcal T_n^+, \mathcal T_n\right)\right)^p \right]  \leq \bE\left[ \left({\rm ht} \left(\mathcal T_{\max}\right)\right)^p \right] \bE \left[ \sum_{j \geq 1} \xi_{j}^{p \beta}\right]^n \rightarrow \infty$$ 
as $n \rightarrow \infty$. By Corollary \ref{limit tree}, we have $\mathcal{T}^\circ_n\rightarrow\cT$ and hence $\mathcal T_n \rightarrow \mathcal T$ in probability as $n \rightarrow \infty$ in the marked Gromov--Hausdorff topology. Uniqueness follows from the attraction property.
\end{pfofthm42}

\bibliographystyle{abbrv}
\bibliography{stablerde2}

\end{document}